\documentclass[a4paper,11pt,twoside]{amsart}
\usepackage[english]{babel}
\usepackage[utf8]{inputenc}

\usepackage[a4paper,inner=3cm,outer=3cm,top=4cm,bottom=4cm,pdftex]{geometry}
\usepackage{fancyhdr}
\usepackage{titlesec}
\usepackage{enumerate}
\usepackage{color}
\usepackage{bold-extra}
\titleformat{\section}{\normalfont\scshape\centering}{\thesection}{1em}{}
\titleformat{\subsection}{\bfseries}{\thesubsection}{1em}{}
  
\usepackage{comment}
\usepackage{graphics}
\usepackage{aliascnt}
\usepackage[pdftex,citecolor=green,linkcolor=red]{hyperref}

\usepackage{amsmath}
\usepackage{amsfonts}
\usepackage{amssymb}
\usepackage{amsthm}
\usepackage{comment}
\usepackage{mathtools}

\newtheorem{theorem}{Theorem}[section]
\newtheorem{corollary}[theorem]{Corollary}
\newtheorem{hypothesis}[theorem]{Hypothesis}
\newtheorem{question}[theorem]{Question}
\newtheorem{lemma}[theorem]{Lemma}
\newtheorem{proposition}[theorem]{Proposition}
\theoremstyle{definition}
\newtheorem{definition}[theorem]{Definition}
\newtheorem{remark}[theorem]{Remark}

\numberwithin{equation}{section}

\newcommand{\ord}{\textup{ord}}
\newcommand{\Gal}{\textup{Gal}}

\newcommand{\Li}{\textup{Li}}
\newcommand{\lcm}{\textup{lcm}}
\newcommand\blfootnote[1]{%
  \begin{NoHyper}%
  \renewcommand\thefootnote{}\footnote{#1}%
  \addtocounter{footnote}{-1}%
  \end{NoHyper}%
}

\author[Olli J\"arviniemi]{Olli J\"arviniemi}
\address{Department of Mathematics and Statistics, University of Turku, 20014 Turku, Finland}
\email{olli.a.jarviniemi@utu.fi}

\author[Joni Ter\"{a}v\"{a}inen]{Joni Ter\"{a}v\"{a}inen}
\address{Department of Mathematics and Statistics, University of Turku, 20014 Turku, Finland}
\email{joni.p.teravainen@gmail.com}

\title{Composite values of shifted exponentials}

\date{}

\begin{document}
\begin{abstract}
A well-known open problem asks to show that $2^n+5$ is composite for almost all values of $n$. This was proposed by Gil Kalai as a possible Polymath project, and was first posed by Christopher Hooley. We settle this problem assuming GRH and a form of the pair correlation conjecture. We in fact do not need the full power of the pair correlation conjecture, and it suffices to assume an inequality of Brun--Titchmarsh type in number fields  that is implied by the pair correlation conjecture.  Our methods apply in fact to any shifted exponential sequence of the form $a^n-b$ and show that, under the same assumptions, such numbers are $k$-almost primes for a density $0$ of natural numbers $n$. Furthermore, under the same assumptions we show that $a^p-b$ is composite for almost all primes $p$ whenever $(a, b) \neq (2, 1)$.
\end{abstract}

\maketitle

\section{Introduction}
\label{sec:intro}

\blfootnote{Keywords: Exponential sequences, Chebotarev density theorem, sieve methods.}It is a notorious open question to determine whether a shifted exponential sequence $a^n-b$ for given $a>1$ and $b\in \mathbb{Z}\setminus\{0\}$ produces infinitely many primes as $n$ ranges over the positive integers $\mathbb{N}$. One expects that this should be the case whenever there is no obvious reason for $a^n-b$ to be composite for all large $n$, the obvious reasons being that $a^n-b$ either has a fixed prime divisor or that it factors as the result of a polynomial identity. This leads to the following question.

\begin{question}\label{conj1} Let $a>1$ and $b\neq 0,-1$ be integers. Assume that for every $q\geq 1$ there exists $1\leq r\leq q$ such that:
\begin{itemize}
\item [(i)] The sequence $a^{qn+r}-b$ has no fixed prime divisor;
\item [(ii)] There is no $m\geq 2$, $m\mid q$ such that $ba^{-r}$ is an $m$th power of a rational number, and further if $4\mid q$ there is no $c\in \mathbb{Q}$ such that $ba^{-r}=-4c^4$);
\end{itemize}

Then does the sequence $a^n-b$ contain infinitely many primes?
\end{question}
Here case (ii) corresponds to the well-known fact that the binomial $x^n-d\in \mathbb{Q}[x]$ is reducible if and only if $d$ is an $m$th power of a rational number for some $m\geq 2$, $m\mid n$, or $4\mid n$ and $d=-4c^4$ for some rational $c$. The exclusion of $b=-1$ stems from the fact that for $b=-1$ one easily sees that the only possible primes in the sequence are of the form $a^{2^m}+1$, and as is discussed below, probabilistic heuristics suggest that there are only finitely many such primes (even though (i) and (ii) hold for $(a,b)=(2,-1)$, say). See Section~\ref{sec:necessary} for further discussion on the necessity of the conditions.

Question~\ref{conj1} is closely connected with two conjectures that are among the oldest in number theory: the existence of Mersenne primes and Fermat primes. Mersenne primes are primes of the form $2^{p}-1$ with $p$ a prime; one easily sees  that any prime of the form $2^n-1$ must be of this form. Many of the largest known primes are Mersenne primes. Fermat primes are primes in the sequence $2^{2^m}+1$; again, all the primes of the form $2^n+1$ are easily seen to have this shape. The first five Fermat numbers ($m=0,1,2,3,4$) are all prime, but extensive numerical searches have produced no further Fermat primes. A widely-believed conjecture (supported by probabilistic arguments; see~\cite[Problem A3]{guy}) is that there are infinitely many Mersenne primes but only finitely many Fermat primes. These two assertions seem to be well beyond reach of all known methods.

Although little is known about the primality of $2^n\pm 1$, it is noteworthy that $2^n\pm 1$ can be prime only for a set of integers $n$ of natural density $0$; this follows immediately from the form that the exponents of the Mersenne and Fermat primes take. In light of this, one conjectures more generally that $a^n-b$ is prime for a natural density $0$ of natural numbers $n$. Even this is still an open problem. The problem, in the concrete special case of the sequence $2^n+5$, was suggested by Gil Kalai in~\cite{kalai} as a possible Polymath project (with the comment that it ``might be too hard''). The problem was originally studied by Christopher Hooley in his book~\cite{hooleybook},  and was popularized by Peter Sarnak\footnote{Sarnak remarks in~\cite{sarnak} that ``Even a problem like $2^n+5$ being
composite for almost all $n$ is very problematic (Hooley).''}~\cite{sarnak}. 

In the discussion in~\cite{kalai}, it was suggested that in order to make any progress on the compositeness of $2^n+5$ one should assume GRH.  Our first main result confirms that $2^n+5$ is indeed composite for almost all $n$, as well as the analogous result for general shifted exponential sequences $a^n-b$, assuming GRH and an additional Brun-Titchmarsh type inequality for primes satisfying conditions of form ``$p \equiv 1 \pmod{m}$, $c$ is a perfect $m'$th power modulo $p$''. We postpone detailed discussion of the hypotheses to Subsection~\ref{subsec:conjecture}.

\begin{hypothesis}[A Brun--Titchmarsh estimate on average]
\label{hyp}
Let $a$ and $b$ be fixed coprime integers with $|a|, |b| > 1$. The following holds for any small enough $0<\epsilon'<\epsilon$:

Let $m$ be a positive integer and let $m'$ be a prime with $m'\mid m$ and $m' \in [(\log m)^{\epsilon'}, m^{\epsilon}]$. Then, uniformly for $y \ge m\exp((\log m)^{1/2})$,  we have
\begin{align}\label{e1}\begin{split}
&\sum_{r=0}^{m'-1}|\{p \le y : p \equiv 1 \pmod{m}, ba^{-r} \text{ is perfect } m'\text{th power mod } p\}|^2 \\
&\ll_{\epsilon, \epsilon'} \frac{y^2}{\phi(m)^2(\log y)^2}\max\left\{\frac{1}{(\log y)^2},\frac{1}{(m')^{\epsilon}}\right\}.
\end{split}
\end{align}
\end{hypothesis}

We then state our main results.

\begin{theorem}
\label{thm:main}
Assume GRH and Hypothesis~\ref{hyp}. Let $a>1$ and $b$ be integers. The natural density of positive integers $n$ such that $a^n - b$ is composite is $1$.
\end{theorem}

\begin{remark}
Hooley had shown in~\cite{hooleybook} that $2^n-b$ is composite for almost all $n$ assuming GRH and an essentially self-serving hypothesis. Our task in this paper is to remove this self-serving hypothesis and replace it with a more natural one.
\end{remark}

\begin{remark}
One motivation for Sarnak to popularize Hooley's problem was that Bourgain, Gamburd and Sarnak~\cite[Theorem 3]{bgs-markoff} managed to settle the analogous problem for the Markoff numbers. Markoff triples are defined as the solutions $(x_1,x_2,x_3)\in \mathbb{N}^3$ to the Diophantine equation $x_1^2+x_2^2+x_3^2=3x_1x_2x_3$, and the  Markoff numbers are the increasing sequence formed by the largest coordinates of Markoff triples $(x_1,x_2,x_3)$ (with multiplicities). By the result of~\cite{bgs-markoff} almost all Markoff numbers are composite, and their number up to $X$ is $\asymp (\log X)^2$. Therefore, the sequences $a^n-b$ are even sparser.

Sarnak~\cite{sarnak} also connected Hooley's problem to the affine sieve developed in~\cite{salehi-sarnak},~\cite{bgs}; the affine sieve theorem of Salehi Golsefidy and Sarnak~\cite{salehi-sarnak} applies to counting almost primes lying in an orbit of a group of affine linear transformations under the assumption that the Zariski closure of the group is \emph{Levi-semisimple}. The authors of~\cite{salehi-sarnak} present a heuristic argument for the necessity of the Levi-semisimplicity condition, which boils down to understanding almost primality questions for shifted exponential functions, such as the one above. 
\end{remark}

In addition to showing compositeness of $a^n-b$ for almost all $n$, we are more generally able to show that $a^n-b$ is not a $k$-almost prime (a number with at most $k$ prime factors) for a density $1$ set of natural numbers $n$. We prove this in the following strong sense.

\begin{theorem}
\label{thm:almostprime}
Assume GRH and Hypothesis~\ref{hyp}. Let $a>1$ and $b\neq 0$ be coprime integers. Then there exists a constant $c=c_{a,b}>0$ such that the natural density of positive integers $n$ such that $\omega(a^n - b)\geq c \log \log n$ is $1$.
\end{theorem}

Here, as usual, $\omega(n)$ denotes the number of distinct prime factors of $n$.

We will in fact prove that the number of prime divisors of $a^n - b$ which are less than $\sqrt{n}$ is almost always $\gg \log \log n$. On the other hand, $a^n-b$ should of course have a lot of prime factors $p>\sqrt{n}$ as well, but our method does not apply to detecting these large factors.

One naturally wonders whether one can say something about the compositeness of $a^n-b$ even when $n$ is restricted to an interesting subset of natural numbers, particularly the primes. In the case of Mersenne primes, which are of the form $2^p-1$ with $p$ prime, it appears unknown that there are even infinitely many \emph{composite} numbers in this sequence. Nevertheless, we are able to apply our methods also to the sequence $a^p-b$ with prime exponents, as long as we are not in the case $a=2$, $b=1$ that corresponds to Mersenne primes. 

\begin{theorem}
\label{thm:primeexp}
Assume GRH and Hypothesis~\ref{hyp}. Let $a>1$ and $b$ be integers with $(a,b)\neq (2,1)$. The relative density of primes $p$ for which $a^p-b$ is composite is $1$.
\end{theorem}

\subsection{The hypotheses}\label{subsec:conjecture}

In this subsection we discuss the hypotheses appearing in our results, namely the generalized Riemann hypothesis (GRH) and Hypothesis~\ref{hyp}. We also connect Hypothesis~\ref{hyp} to the pair correlation conjecture (PCC) stated below.

The precise form of the generalized Riemann hypothesis (GRH) that we need is as follows. It only involves certain special field extensions, which we call Kummer-type extensions.

\begin{definition} We say that a field extension $K/\mathbb{Q}$ is a \emph{Kummer-type extension} if $K$ is of the form $K=\mathbb{Q}(\zeta, a_1^{1/m_1},\ldots, a_k^{1/m_k})$, where $\zeta$ is a primitive root of unity of some order $m$, $a_i>1$ are integers and $m_i\mid m$ for all $i\leq k$.
\end{definition}

\textbf{GRH:} For any Artin $L$-function\footnote{See~\cite{neukirch} for an introduction to the theory of Artin $L$-functions.} associated with a Kummer-type extension, its zeros in the critical strip $\textnormal{Re}(s)\in (0,1)$ all lie on the critical line $\textnormal{Re}(s)=1/2$.\\

We then give a couple of remarks on Hypothesis~\ref{hyp}.

\begin{itemize}
\item The condition ``$p \equiv 1 \pmod{m}$, $ba^{-r}$ is perfect $m'$th power mod $p$'' is naturally viewed as ``$p$ splits completely in $\mathbb{Q}(\zeta_m, (ba^{-r})^{1/m'})$''. Hence Hypothesis~\ref{hyp} asks for a bound on the number of totally split primes in a number field.
\item One may view the hypothesis as a generalization of the classical Brun--Titchmarsh inequality (with extra averaging), which states that, uniformly in the region $y>m$, we have
\begin{align*}
|\{p \le y : p \,\,\textnormal{splits completely in}\,\, \mathbb{Q}(\zeta_m)\}|\ll \frac{y}{\phi(m)\log(y/m)}.  
\end{align*}
This, combined with the fact that $p$ can only split completely in $O_{a,b}(1)$ extensions of the form $\mathbb{Q}(\zeta_m,(ba^{-r})^{1/m'})$, gives an upper bound for the left-hand side of~\eqref{e1} of the form $\ll y^2/(\phi(m)^2(\log(y/m))^2)$, and so this trivial bound falls short of Hypothesis~\ref{hyp} by a few logarithms at most. 

\item One would more strongly expect for the left-hand side of~\eqref{e1} a bound of 
\begin{align}\label{e101}
\ll_{\epsilon',A}\frac{y^2}{\phi(m)^2(\log y)^2}\max\left\{\frac{1}{(\log y)^{A}},\frac{1}{m'}\right\}   
\end{align}
(and also without the primality assumption on $m'$), since one can prove under our assumptions that $[\mathbb{Q}(\zeta_m, (ba^{-r})^{1/m'} : \mathbb{Q}] \asymp \phi(m)m'$, and the probability of a prime splitting completely in $K/\mathbb{Q}$ should be comparable to $1/[K:\mathbb{Q}]$ with a wide range of uniformity, as is known in the case of cyclic extensions.

In fact, under GRH (which we are assuming in any case) the Chebotarev density theorem gives an asymptotic estimate of the form
\begin{align}\label{e104}
\frac{y}{[\mathbb{Q}(\zeta_m, (ba^{-r})^{1/m'}) : \mathbb{Q}]}+O_{\eta}((ym)^{1/2+\eta})    
\end{align}
for the number of primes splitting completely in $\mathbb{Q}(\zeta_m,(ba^{-r})^{1/m'})$, and this gives~\eqref{e101} for $y\geq m^{3+O(\eta)}$. Therefore, the essence of Hypothesis~\ref{hyp} is that we can also deal with the range where $y \approx m^{1+\delta}$.

\item Hypothesis~\ref{hyp} should be true in the full range $m'\in [1,m]$,. However, as is clear from the bound given by the Chebotarev density theorem,~\eqref{e1} becomes much more challenging to prove as the quantity $\phi(m)m'$ increases. Therefore, in our arguments we exercise great care to minimize the range of $m'$ in which we assume~\eqref{e1}, in the hope that the case of small $m'$ would turn out to lie less deep than the case of large $m'$. 

\item The bound~\eqref{e1} resembles a Brun--Titchmarsh estimate in number fields, as our naming of the hypothesis suggests. Unfortunately, the current Brun--Titchmarsh analogues of the Chebotarev density theorem are too weak to imply the Hypothesis. For example, the results of~\cite{thorner} or~\cite{debaene} are not applicable when $m' \gg \log m$.

\end{itemize}

Next we connect Hypothesis~\ref{hyp} to the pair correlation conjecture (PCC), which in fact implies our hypothesis.

\textbf{PCC:} Let $K/\mathbb{Q}$ be a Kummer-type extension, and let $L(s,\chi, K/\mathbb{Q})$ be the Artin $L$-function associated with an irreducible character $\chi$ of $\textnormal{Gal}(K/\mathbb{Q})$. Define the pair correlation function
\begin{align*}
F(X,T;\chi):=\sum_{-T\leq \gamma_1,\gamma_2\leq T}w(\gamma_1-\gamma_2)X^{(\gamma_1-\gamma_2)},
\end{align*}
where $\gamma_1,\gamma_2$ run through the imaginary parts of zeros of $L(s,\chi,K/\mathbb{Q})$ on the critical line $\textnormal{Re}(s)=1/2$ and $w(u):=\frac{4}{4+u^2}$ is a weight function (note that $F(X,T;\chi)$ is always real-valued). Further define the conductor $\mathcal{A}_{\chi}(T):=A_{\chi}T^{\chi(1)}$, where $A_{\chi}:=d_K^{\chi(1)}\textnormal{Norm}(\mathfrak{f}(\chi))$, where $d_K$ is the degree of the field extension $K/\mathbb{Q}$ and $\mathfrak{f}(\chi)$ is the Artin conductor of $\chi$. With this notation, for any $C>0$ and $1\leq X\leq T^{\chi(1)C}$, we have 
\begin{align*}
F(X,T;\chi)\ll_{C}\chi(1)^{-1}T\log \mathcal{A}_{\chi}(T).
\end{align*}

\textbf{Remarks.}

\begin{itemize}

\item This PCC conjecture is denoted $\textnormal{PCC}(\chi,\chi(1),\chi(1)^{-1},1)$ by M. R. Murty, V. K. Murty and Wong in~\cite[Conjecture 3.2]{murty-wong} and is a special case of a more general conjecture $\textnormal{PCC}(\chi,m_{\chi},c_{\chi},r)$ stated there. They state that this conjecture first arose in unpublished work of M. R. Murty and V. K. Murty 20 years earlier.

\item Evidence for PCC includes the result~\cite[Proposition 3.1]{murty-wong}, which unconditionally gives
\begin{align*}
F(X,T;\chi)\ll_{A}T(\log \mathcal{A}_{\chi}(T))^2,    
\end{align*}
so the content of the conjecture is in reducing the power of logarithm here. As noted in~\cite{murty-wong}, the form of PCC stated here is weaker than some other forms of the pair correlation conjecture in the aspect that they further require an asymptotic formula for $F(X,T;\chi)$. In particular, Montgomery's ~\cite{montgomery} classical pair correlation conjecture for the Riemann zeta function (which corresponds to $K=\mathbb{Q}$ and $\chi\equiv 1$) predicts not only an upper bound of $\ll T\log T$ for the pair correlation function of $\zeta(s)$ but more strongly an asymptotic formula of the same order of magnitude. Montgomery in fact proved this asymptotic formula unconditionally when $C\leq 1$ in the notation of the PCC conjecture above. See also~\cite{odlyzko} for numerical evidence for the pair correlation conjecture in the case of the Riemann zeta function,~\cite{rudnick-sarnak} for an asymptotic formulation for all automorphic $L$-functions, and~\cite{murty-perelli} for the formulation for any Dirichlet series in the Selberg class.
\end{itemize}

\begin{theorem}\label{thm:pcc}
PCC for all Kummer-type extensions implies Hypothesis~\ref{hyp}.
\end{theorem}

In fact, as will be clear from the proof, PCC implies a much stronger version of Hypothesis~\ref{hyp}, where we can take $\epsilon=1$.

Let us briefly comment on how our two assumptions, GRH and Hypothesis~\ref{hyp}, enter the proof. The GRH assumption is clearly necessary, as will be seen from formula~\eqref{equ1} below (which is currently provable under GRH but not unconditionally). Also, Chebotarev's density theorem is a key tool in our proofs, and in order for its error bounds to be strong enough we need to assume GRH. 

Hypothesis~\ref{hyp} is put into use in only one part of the proof. The idea is roughly as follows. For $p$ a prime, let $\ell(p)$ denote the least positive integer $l$ such that $a^l \equiv b \pmod{p}$ (if such an integer $l$ exists).
In the course of the proof, we need upper bounds for the number of primes $p\leq y$ that satisfy $p\equiv 1\pmod m$ and $\ell(p)\equiv r\pmod m$, for $y$ just a bit larger than $m$ (say $y\approx m^{1+\epsilon}$ or $y \approx m\exp((\log m)^{1/2})$).  This condition can be naturally interpreted in terms of splitting of primes in certain Kummer-type extensions, and the Chebotarev density theorem tells us that such $p$ are roughly\footnote{One does not necessarily have exact equidistribution (in sense of the natural density) even for $m$ fixed. This technical detail is discussed in more detail later on.} equidistributed among the different values of $r$, at least for $y$ large in terms of $m$. The small values of $y$ (compared to $m$) are more difficult. However, as we only need upper bounds rather than actual equidistribution, we can afford to weaken the $\ell(p)\equiv r\pmod m$ condition to $\ell(p)\equiv r\pmod{m'}$ for $m'\mid m$ being a relatively small divisor of the modulus.  Now, Hypothesis~\ref{hyp} essentially states that the values $\ell(p) \pmod{m'}$ for the primes $p\leq y$ with $p\equiv 1\pmod m$ are not clustered into a too small a set of residues, which is what we need in the proof.

In the case $y$ is very close to $m$ ($y \le m\exp((\log m)^{1/2})$), we may neglect the condition $\ell(p) \equiv r \pmod{m}$ and merely work with the congruence conditions $p \equiv 1 \pmod{m}$. Such a crude estimate does not provide estimates good enough for our purposes in regions larger than $y \le m\exp((\log m))^{1/2})$, so for larger values of $y$ one has to somehow account for the condition $\ell(p) \equiv r \pmod{m}$. As Hypothesis~\ref{hyp} is essentially our only tool for such considerations for $y$ relatively close to $m$, we require it to be applicable in the region $y \ge m\exp((\log m)^{1/2})$.

\subsection{Unconditional work}

Not much is known about the composite values of $a^n-b$ unconditionally (except in the trivial case where conditions (i) and (ii) in Question~\ref{conj1} are not satisfied). Note that if the congruence $a^n-b\equiv 0\pmod{p}$ has a solution $\ell(p)$ (which we assume to be the minimal positive solution),  the general solution to the congruence is given by $n\equiv \ell(p)\pmod{\ord_p(a)}$. The set of such $n$ has natural density $1/\ord_p(a)$. Thus, by the Borel--Cantelli lemma, a necessary (but not sufficient) condition for almost all numbers of the form $a^n-b$ to be have a prime factor $\geq K$ for every given $K$ is that
\begin{align}\label{equ1}
\sum_{\substack{p\\ p\mid a^n-b \,\, \textnormal{for some}\,\, n}} \frac{1}{\ord_p(a)}=\infty.
\end{align}
This estimate has not been proved unconditionally, meaning that if one wants to make progress on the composite values of $a^n-b$, one must assume a conjecture that implies~\eqref{equ1}. It seems that the best known unconditional estimate for the number of primes $p\leq x$ dividing some element of the sequence $a^n-b$ is $\gg \log x$ (in contrast with the conjectured order of magnitude of $\asymp \pi(x)$), proved by M. R. Murty, S\'eguin and Stewart~\cite{mss}; this is a bound that is unfortunately far too weak to imply~\eqref{equ1}.

It is also worth noting that Hooley showed unconditionally in~\cite{hooleybook} that the sequence $n\cdot 2^n+1$ produces primes (the Cullen primes) for a density zero of natural numbers $n$; as pointed out by Elsholtz in~\cite{elsholtz}, this is a lot easier than the corresponding question for the sequence $a^n-b$, since it is easier to control the distribution of Cullen primes in residue classes (regarding the problem of the sequence $a^n-b$, he states that ``current methods do not work'' for it). Hooley's method was further refined by Rieger~\cite{rieger} to yield that $p\cdot 2^p+1$ is prime for a relative density $0$ of primes $p$.

\subsection{Connection with the Artin primitive root conjecture}

The most natural conjecture known to imply~\eqref{equ1} is GRH; this implication follows from work of Moree and Stevenhagen~\cite{moree-stevenhagen}. Their work is related to Hooley's~\cite{hooley} work on Artin's primitive root conjecture, which  is the statement that $\ord_p(a)=p-1$ for infinitely many primes, whenever $a>1$ is a fixed integer that is not a perfect power. A wide generalization of Artin's conjecture was proved by Lenstra~\cite{lenstra}, again under GRH, and in fact our proof of Theorem~\ref{thm:main} involves (among other things) adapting his work to sequences of the form $a^n-b$. This produces the following intermediate result containing~\eqref{equ1} (by Mertens' theorem), which may be of independent interest.

\begin{corollary}\label{cor1} Assume GRH. Let $a>1$ be an integer that is not a perfect power, and let $b\neq 0$. Let $k\geq 1$ be the largest integer for which $b$ is a perfect $k$th power (if $|b|=1$, we define $k=1$). Then the set of primes $p$ satisfying $p\mid a^n-b$ for some $n$ and $\ord_p(a)=\frac{p-1}{2k}$ possesses a natural density, which is positive and can be computed explicitly. 
\end{corollary}

\subsection{Acknowledgments} The authors thank the anonymous referees for numerous helpful corrections. The authors thank Peter Sarnak and Jesse Thorner for helpful discussions. The first author was supported by the Emil Aaltonen foundation and worked in the Finnish Centre of Excellence in Randomness and Structures (Academy of Finland grant no. 346307). The second author was supported by a Titchmarsh Fellowship and by European Union's Horizon
Europe research and innovation programme under Marie Sk\l{}odowska-Curie grant agreement No 101058904.

\section{Notation, conventions, and some preliminaries}
\label{sec:notation}

Without further mention, we assume GRH in the lemmas and theorems that follow. We will assume Hypothesis~\ref{hyp} only in the proof of Lemma~\ref{lem:hypothesis} below.

We may assume in the proofs of Theorems~\ref{thm:main} and~\ref{thm:almostprime} that $a$ is not a perfect power, as the case of $a$ being a perfect power immediately reduces to the generic case. If $|b| \le 1$, the conclusion of Theorems~\ref{thm:main} and~\ref{thm:primeexp} is trivial, since $a^{m}-1\mid a^n-1$ for $m\mid n$ and $a^m+1\mid a^n+1$ for $m\mid n$ and $m$ odd. Thus we may assume $|b| > 1$ when proving these theorems. We may also assume that $(a, b) = 1$, since otherwise $a^n-b$ is composite.

If $p$ is a prime not dividing $a$, we denote by $\ord_p(a)$ the least positive integer $e$ such that $a^e \equiv 1 \pmod{p}$.

We denote by $\zeta_k$ any primitive $k$th root of unity.

We let $\tau(n), \phi(n)$ denote the divisor and Euler phi functions, respectively. We let $(a,b)$ stand for the greatest common divisor of $a$ and $b$ and $\textnormal{lcm}(a,b)$ for their least common multiple. By $\textnormal{rad}(n)$ we denote the product of the prime factors of $n$.

The natural density of a set $S \subset \mathbb{N}$ is defined as the limit
$$d(S):=\lim_{x \to \infty} \frac{|S \cap [1, x]|}{x},$$
provided that the limit exists. We similarly define the relative density of a set $A \subset \mathbb{N}$ in the set $B \subset \mathbb{N}$ as
$$\lim_{x \to \infty} \frac{|A \cap [1, x]|}{|B \cap [1, x]|},$$
assuming that $A \subset B$ and that the limit exists. The density $d_{\mathbb{P}}(\mathcal{P})$ of a subset $\mathcal{P}$ of the primes $\mathbb{P}$ is to be considered as its relative density in the set of primes.

For any subset $S$ of the primes, we define $\pi_S(x):=|\{p\leq x:\,\, p\in S\}|$.

Given an extension $L/K$ of number fields, a prime $\mathfrak{p}$ of $\mathcal{O}_K$, and a prime $\mathfrak{P}$ of $\mathcal{O}_L$ lying above $\mathfrak{p}$, we define the Artin symbol $\left(\frac{L/K}{\mathfrak{P}}\right)$ as the unique element $\sigma\in \textnormal{Gal}(L/K)$ satisfying
\begin{align*}
\sigma(\alpha)\equiv \alpha^{\textnormal{Norm}(\mathfrak{p})}\pmod{\mathfrak{P}}    
\end{align*}
for all $\alpha \in L$. Further, if $K=\mathbb{Q}$ and $p$ is a rational prime unramified in $L$, we define $\left(\frac{L/\mathbb{Q}}{p}\right)$ as the conjugacy class of possible values of  $\left(\frac{L/\mathbb{Q}}{\mathfrak{P}}\right)$ with $\mathfrak{P}$ lying above $p$; it can be checked that such values do indeed form a conjugacy class. (This definition could be extended to $K \neq \mathbb{Q}$ as well, but we will only need the case $K = \mathbb{Q}$.)

The prime $p$ is said to split completely in $L$ if $\left(\frac{L/\mathbb{Q}}{p}\right)$ is the conjugacy class consisting of the identity of $\Gal(L/\mathbb{Q})$.

We use the following version of the Chebotarev density theorem, due to Serre~\cite{serre} and conditional on GRH (this improves on work of Lagarias and Odlyzko~\cite{lo}).
 
\begin{lemma}[Chebotarev density theorem]\label{lem:chebotarev}
Assume GRH. Let $K$ be a finite Galois extension of $\mathbb{Q}$ with Galois group $G$. Let $C$ be a conjugacy class of $G$. The number of (rational) unramified primes $p$ with $p \le x$ and Artin symbol $\left(\frac{K/\mathbb{Q}}{p}\right) = C$ is
$$\pi_C(x)=\frac{|C|}{|G|}\Li(x) + O\left(\frac{|C|}{|G|}\sqrt{x}(\log \textnormal{disc}(K/\mathbb{Q}) + [K : \mathbb{Q}]\log x)\right),$$
where $\textnormal{disc}(K/\mathbb{Q})$ is the discriminant of $K/\mathbb{Q}$ and $\Li(x)=\int_{2}^x \frac{dt}{\log t}$. 
\end{lemma}

\begin{proof}
This is~\cite[Th\'eor\`eme 4]{serre}.
\end{proof}

\section{Overview of the method}
\label{sec:overview}

In this section, we give a sketch of the proof method of Theorem~\ref{thm:main}; the actual details in subsequent sections are slightly different and more complicated, but the purpose of this section is just to illustrate the approach.

As already mentioned, we may assume that $a$ is not a perfect power. Let 
\begin{align}
P_a=\{p\in \mathbb{P}:\,\, a^x\equiv b\pmod p\,\, \textnormal{for some}\,\, x\};    
\end{align}
since $P_a$ contains those primes $p$ for which $a$ is a primitive root, under GRH the set $P_a$ contains a positive proportion of the primes. For each $p \in P_a$, there exists a unique integer $\ell(p) \in [1, \ord_p(a)]$ such that $a^{\ell(p)} \equiv b \pmod{p}$. Now, for any integer $n \equiv \ell(p) \pmod{\ord_p(a)}$, we have $a^n \equiv b \pmod{p}$, implying that $a^n - b$ is not prime (except possibly for $n = \ell(p)$). The goal is to prove that the density of integers covered by such arithmetic progressions $\ell(p)\pmod{\ord_p(a)}$ is $1$. Thus the statement is that the residue classes $\ell(p)\pmod{\ord_p(a)}$ for $p\leq P\to \infty$ are nearly (up to a vanishing proportion of numbers) a \emph{covering system} of the integers. This is generally speaking a rather delicate condition; see~\cite{ffkpy} and~\cite{balister} for work on  finite sets of congruences covering all but an $\varepsilon$-proportion of the integers\footnote{Note that whether or not residue classes $a_p\pmod {p-1}$ cover almost all numbers depends heavily on the values of the $a_p$; for example, it is known that the proportion of integers covered by $0\pmod {p-1}$ for $p>P$ approaches $0$ as $P\to \infty$. Also if $\mathcal{P}$ is any subset of the primes with sum of reciprocals less than $\frac{1}{2}$, say, then $a_p\pmod{p-1}$ for $p\in \mathcal{P}$ leaves uncovered a positive proportion of all numbers.}.

If the conditions $x \equiv \ell(p) \pmod{\ord_p(a)}$ were independent of each other (which would be the case by the Chinese reminder theorem if the moduli $\ord_p(a)$ and $\ord_q(a)$ were coprime for all $p \neq q$), the density of integers not covered by such congruence conditions would be
$$\prod_{p \in P_a} \left(1 - \frac{1}{\ord_p(a)}\right).$$
Under GRH the relative density of $P_a$ inside the primes is positive, so by Mertens' theorem the above product evaluates to $0$. 

However, since $\ord_p(a)$ and $\ord_q(a)$ are typically \emph{not} coprime\footnote{Note that it does not help to restrict to a subset of primes for which $\ord_p(a)/2$ are pairwise coprime, since such a set always has finite sum of reciprocals. If for example we restrict to primes $p$ such that $(p-1)/2$ is also prime, then it is known that such primes have bounded sum of reciprocals (even though their infinitude is not known).}, we have to take into account the dependencies between the conditions imposed by different primes. For this we utilize the second moment method.

\begin{lemma}\label{lem:2moment}
Let $(\Omega, \textnormal{Pr})$ be any finitely additive probability space, and let $A_1,A_2,\ldots,A_n\in \Omega$ be events. Denote
\begin{align}\label{e3}
\mu:=\sum_{i=1}^n \Pr(A_i).
\end{align}
Then, for any $\varepsilon>0$,  we have 
\begin{align}\label{e2}\Pr\left(x\in \Omega:\,\, ||\{i\in [1,n]:\,\, x\in A_i\}|-\mu|\geq \varepsilon \mu\right) \le \varepsilon^{-2}\left(\sum_{i = 1}^n \sum_{j = 1}^n \Pr(A_i \cap A_j)/\mu^2-1\right).
\end{align}
\end{lemma}

\begin{proof}
This follows by writing the left-hand side as 
\begin{align*}
 \Pr\left(x\in \Omega:\,\, \left|\sum_{i=1}^n 1_{A_i}(x)-\mu\right|\geq \varepsilon \mu\right)   
\end{align*}
and using Chebychev's inequality to upper-bound this as 
\begin{align*}
\varepsilon^{-2}\mu^{-2}\mathbf{E}\left|\sum_{i=1}^n 1_{A_i}(x)-\mu\right|^2,    
\end{align*}
and then expanding out the square.
\end{proof}

\begin{remark}
If the $A_i$ were pairwise independent, the right-hand side of~\eqref{e2} would be $0$. More generally, if we show that the $A_i$ are approximately independent, then the right-hand side of~\eqref{e2} is small.   
\end{remark}

In our case, the events $A_p$ correspond to a ``random'' integer being congruent to $\ell(p) \pmod{\ord_p(a)}$, where $p$ ranges over $P_a \cap [1, x]$, from which we get $\Pr(A_p) = 1/\ord_p(a)$. (See Section \ref{sec:chungBound} for a precise formulation of this idea.) Then the expected value $\mu$ appearing in~\eqref{e3} is, more or less, equal to
\begin{align}\label{equ2}
\sum_{\substack{p\leq x\\p\in P_a}} \frac{1}{\ord_p(a)}.  
\end{align}
For reasons explained below, we will in fact restrict to a subset of $P_a$ where $\ord_p(a) \gg p-1$, which moreover has a positive relative density inside the primes. This then allows for an asymptotic evaluation of \eqref{equ2}. We now turn to the sum of the pairwise intersections in~\eqref{e2}, which are much more difficult to analyze.

Let $p$ and $q$ be primes, and write $(\ord_p(a), \ord_q(a)) = m$. The system $y \equiv \ell(p) \pmod{\ord_p(a)}, y \equiv \ell(q) \pmod{\ord_q(a)}$ has a solution if and only if $m \mid \ell(p) - \ell(q)$. If a solution exists, then $\Pr(A_p \cap A_q) = \frac{m}{\ord_p(a)\ord_q(a)}$, and otherwise $\Pr(A_p \cap A_q) = 0$. Thus
\begin{align}\label{equ3}
\sum_{\substack{p\leq x\\p\in P_a}} \sum_{\substack{q \leq x\\q\in P_a}} \Pr(A_p \cap A_q) = \sum_{m \le x} \sum_{\substack{p, q \le x\\ p,q\in P_a\\(p-1, q-1) = m\\ m \mid \ell(p) - \ell(q)}} \frac{m}{\ord_p(a)\ord_q(a)}.
\end{align}
If the pairs $(\ell(p),\ell(q))$ were equidistributed modulo all $m\leq x$, even when conditioned on the event $(\ord_p(a), \ord_q(a)) = m$, we would see that $A_p$ and $A_q$ are independent ``on average'', resulting in the same asymptotics for~\eqref{equ3} as for~\eqref{equ2}. 

Unfortunately, the set $P_a$ does \emph{not} have the required equidistribution property for all moduli.\footnote{For example, if $p\equiv \pm 3\pmod 8$ and $2^x \equiv 9 \pmod{p}$, then $x$ must be even. More generally, congruence conditions for $p$ give information on quadratic residues modulo $p$, leading to bias even in power-free cases. For example, if $2^x \equiv -3 \pmod{p}$ and $p \equiv 1 \pmod{3}$, $p\equiv \pm 3\pmod 8$, then $-3$ is a quadratic residue $\pmod p$ and $2$ is not, so $x$ must be even.} In Section~\ref{sec:equiLog}, we construct a positive density set $S$ of primes $p$ (depending on $a,b$) for which $a^x \equiv b \pmod{p}$ is solvable and for which the smallest solution $\ell(p)$ of the congruence does enjoy the required equidistribution property for all fixed $m$ by adapting the method of Lenstra~\cite{lenstra}. We then apply Lemma~\ref{lem:2moment} to this subset $S$ of primes in place of $P_a$ (in which case showing the existence and positivity of $d_{\mathbb{P}}(S)$ requires some work).

The main difficulty then is that we need to estimate~\eqref{equ3} in \emph{all} ranges of $m$, not only for $m$ fixed (even though the case of bounded $m$ should give the main term in~\eqref{equ3} and larger $m$ should only contribute an error term). This will be achieved in four steps, carried out in different sections.

\begin{itemize}
\item [(i)] We handle the case of fixed $m$ (say $m\leq N$) in Section~\ref{sec:gcdSmall}, making use of the equidistribution of the set $S$ in residue classes, which in turn is proved in Section~\ref{sec:equiLog}.

\item [(ii)] The contribution of medium-large $m$ (say $N \leq m\leq \max\{p,q\}^{c}$ for some small constant $c>0$) is handled in Subsection~\ref{sec:primeLarge} using the effective version of the Chebotarev density theorem stated in Lemma~\ref{lem:chebotarev}.

\item [(iii)] The case of very large $m$ (say $\max\{p,q\}/ \exp(\log \max\{p,q\})^{1/2})\leq m\leq \max\{p,q\}$) is handled (unconditionally) by applying Selberg's sieve in Subsection~\ref{sec:primeSmall}. The range here is optimal in the sense that this argument (which starts by dropping the condition $m\mid \ell(p)-\ell(q)$) would not work in any larger regime. 

\item [(iv)] The remaining case of large $m$ ($\max\{p, q\}^c \le m \le \max\{p,q\}/ \exp(\log \max\{p,q\})^{1/2})$) is where we utilize Hypothesis~\ref{hyp}. This case is dealt with in Subsection~\ref{sec:primeMedium}.
\end{itemize}

Combining the estimates for these four different regimes will yield Theorem~\ref{thm:main}.

We comment on some difficulties arising with controlling the contribution of large values of $m$. To bound the sum in~\eqref{equ3} one has to bound the number of primes $x/2 \le p \le x$  (or more generally $x_1 \le p \le x_2$) satisfying $\ell(p) \equiv r \pmod{m}$ summed over all $r \pmod{m}$. (The contribution of large $m$ would be too large if the values $\ell(p) \pmod{m}$ attained only a few values, so we cannot just drop the condition on $\ell(p)$ and work only with the congruence conditions $p \equiv 1 \pmod{m}$.) This condition can be naturally interpreted as $p$ splitting completely in a certain field extension, at least for $p\in S$, where $S$ is the set constructed in Section~\ref{sec:equiLog}. 

If $m$ is small (say $m \le x^{1/4 - \epsilon}$), one can apply the GRH-conditional Chebotarev density theorem for all $r$ separately (see case (ii) above). If $m$ is of magnitude $\sqrt{x}$, one can barely apply the Chebotarev density theorem for a single value of $r$ directly, and to control all $m$ values of $r$ one would need a uniform error term. If $m$ is much larger than $\sqrt{x}$, say $x^{3/4}$, the GRH-conditional square root error term is inapplicable. It is here that we use Hypothesis~\ref{hyp}. 

Note that while for fixed $m$ (see case (i)) we need exact (asymptotic) equidistribution, for large values of $m$ we only need that the residues $\ell(p) \pmod{m'}$ do not obtain a small set of values very often. Heuristically each value $\pmod{m'}$ occurs roughly $1/m'$ of the time (even with $m'$ quite large), but we only need that each value occurs at most $\max\{(1/m')^{\epsilon},1/(\log y)^2\}$ of the time. This is indeed what Hypothesis~\ref{hyp} tells us.

\section{Equidistribution of the discrete logarithm}
\label{sec:equiLog}

Let $h$ be the largest integer such that $|b|$ is a perfect power of order $h$. Recall from Section~\ref{sec:notation} that we can assume $a$ is not a perfect power. Let $S\subset \mathbb{P}$ (which depends on $a$ and $b$) be a set of primes defined as
\begin{align*}
S:=\{p\in \mathbb{P}:\,\, p \equiv 1 \pmod{|4abh|}, \,\, a,b\,\, \textnormal{perfect $2h$th powers}\pmod p,\,\, \ord_p(a) = \frac{p-1}{2h}\}.    
\end{align*}
Note that for any $p \in S$ we have $\ord_p(b) \mid \frac{p-1}{2h} = \ord_p(a)$, so the congruence $a^x \equiv b \pmod{p}$ has a solution. 

Let $\ell(p)$ be the smallest nonnegative solution to $a^x \equiv b \pmod{p}$ for $p \in S$. For each $m, d, r \in \mathbb{N}$, let 
\begin{align}\label{equ17}
S_{m, d, r}:=S\cap \{p\in \mathbb{P}:\,\, md \mid \ord_p(a),\,\, \ell(p) \equiv r \pmod{m}\}.
\end{align}
Note then that for $p\in S_{m,d,r}$ any solution to $a^x \equiv b \pmod{p}$ satisfies $x \equiv r \pmod{m}$. Note also that 
\begin{align*}
p\in S_{m,d,r}\Longleftrightarrow p\in S\quad \textnormal{and}\quad  md\mid \ord_p(a)\quad \textnormal{and}\quad  \exists x:\,ba^{-r}\equiv x^{2hm}\pmod p.   
\end{align*}
Indeed, for the forward implication note that if $p\in S$ and $\ell(p)=r+jm$, then $a^{\ell(p)}\equiv b\pmod{p}$ is equivalent to $ba^{-r}\equiv a^{jm}\pmod p$, and here $a$ is a $2h$th power $\pmod{p}$. Conversely, suppose that $p\in S$, $md\mid \ord_p(a)$ and $ba^{-r}\equiv x^{2hm}\pmod p$. Then by the assumption on $\ord_p(a)$ we have $x^{2h}\equiv a^j$ for some $j$, so $ba^{-r}\equiv a^{jm}\pmod p$ for some $j$, and this implies $r+jm\equiv \ell(p)\pmod m$. 

Our goal in this section is to prove the following two lemmas.

\begin{lemma}
\label{lem:denPos}
The relative density $d_{\mathbb{P}}(S)$ exists and is positive.
\end{lemma}

\begin{lemma}
\label{lem:denInv}
For $d\in \mathbb{N}$ and $m \equiv 0 \pmod{|2ab|}$ fixed, the relative density $d_{\mathbb{P}}(S_{m,d,r})$ exists and does not depend on $r$.
\end{lemma}

Before proving these lemmas, we make some preliminary considerations.

We build on the arguments of Lenstra in~\cite{lenstra}. Let \begin{align*}
F_{m, d, r} := \mathbb{Q}(\zeta_{|4abh|}, \zeta_{2hmd}, a^{1/2h}, b^{1/2h}, (ba^{-r})^{1/2hm}).    
\end{align*} 
Clearly if $m\equiv 0\pmod{|2ab|}$, then $\zeta_{|4abh|} \in \mathbb{Q}(\zeta_{2hmd})$. Since $|b|$ is a perfect $h$th power, $|b|^{1/2h} \in \mathbb{Q}(\zeta_{|4b|})$ (indeed, if $c=|b|^{1/h}$, then $\sqrt{c}\in \mathbb{Q}(\zeta_{|4c|})\subset \mathbb{Q}(\zeta_{|4b|})$). Furthermore, $b^{1/2h}$ is either $|b|^{1/2h}$ or $\zeta_{4h}^j|b|^{1/2h}$ for some $j$. In either case, 
$$F_{m, d, r} = \mathbb{Q}(\zeta_{2hmd}, a^{1/2h}, (ba^{-r})^{1/2hm}).$$
For a prime $\ell$, define 
\begin{align*}
q(\ell) := \min\{\alpha\geq 1:\ell^{\alpha}\nmid 2h\}
\end{align*}
to be the smallest power of $\ell$ which does not divide $2h$, and let 
$$L_{\ell} := \mathbb{Q}(\zeta_{q(\ell)}, a^{1/q(\ell)}).$$
We can now characterize the set $S_{m,d,r}$ as follows.

\begin{lemma}
A large enough prime $p$ satisfies $p \in S_{m, d, r}$ if and only if $p$ splits completely in $F_{m, d, r}$ but not in any of the $L_{\ell}$. 
\end{lemma} 

\begin{proof}
We recall a few simple facts from Galois theory. Firstly, by Dedekind's factorization theorem, an unramified prime $p$ splits completely in a field extension $\mathbb{Q}(\alpha)$ if and only if the minimal polynomial $f_{\alpha}(x)\in \mathbb{Q}[x]$ of $\alpha$ factorizes into distinct linear factors $\pmod p$ for $p$ large enough. Secondly, $p$ splits completely in $\mathbb{Q}(\alpha, \beta)$ if and only if $p$ splits completely in both $\mathbb{Q}(\alpha)$ and $\mathbb{Q}(\beta)$. Thirdly, the cyclotomic polynomial $\Phi_n(x)$ factorizes into distinct linear factors $\pmod p$ if and only if $p\equiv 1\pmod n$.

From these three facts and the well-known fact that the compositum of Galois extensions is Galois, we see that a large prime $p$ splits completely in $F_{m, d, r}$ but not in any of the $L_{\ell}$ if and only if
$a$ is a $2h$th power $\pmod p$, $ba^{-r}$ is a $2hm$th power $\pmod p$, $p\equiv 1\pmod{2hmd}$ and $a$ is not a $q(\ell)$th power $\pmod p$ for any $\ell\mid p-1$. The first and last condition are equivalent to $p \in S$, and together with the second and third condition this is equivalent to $p \in S_{m, d, r}$, so the claim follows (note that the special case $m = d = 1, r = 0$ gives a characterization of the set $S$).
\end{proof}

For $n$ a squarefree integer, let $L_n$ be the compositum of $L_{\ell}$ for primes $\ell \mid n$ and let
\begin{align*}
q(n):=\prod_{\substack{\ell\mid n\\\ell \in \mathbb
{P}}}q(\ell),\quad q(1):=1.    
\end{align*}
Define 
\begin{align*}
a_n := \frac{|\{\sigma\in \textnormal{Gal}(F_{m,d,r}L_n/\mathbb{Q}):\,\, \sigma\,\, \textnormal{fixes}\,\, F_{m,d,r},\,\, \forall \ell\mid n:\,\, \sigma\,\, \textnormal{does not fix}\,\, L_{\ell}\}|}{[F_{m,d,r}L_nm:\mathbb{Q}]}.    
\end{align*}
One would now expect that 
\begin{align}\label{equ4}
d_{\mathbb{P}}(S_{m, d, r})=\lim_{k\to \infty}a_{n_k}, 
\end{align}
where $n_k$ is the product of the first $k$ primes. This is because $a_{n_k}$ is by the Chebotarev density theorem equal to the density of primes $p$ which split completely in $F_{m, d, r}$ but in none of $L_{\ell}, \ell \mid n_k$. Formula~\eqref{equ4} was indeed proved by Lenstra~\cite[Theorem 3.1]{lenstra} under GRH. Thus, in particular, the density $d_{\mathbb{P}}(S_{m, d, r})$ exists.

Another relevant result of Lenstra concerns whether $d_{\mathbb{P}}(S_{m, d, r})$ equals to $0$ or not. By~\eqref{equ4}, this corresponds to understanding the behavior of the local densities $a_n$. Clearly $a_n\leq a_m$ for $m\mid n$, so if $a_n = 0$ for some $n$, then $d_{\mathbb{P}}(S_{m, d, r}) = 0$. In fact, the converse is true as well, as shown by Lenstra~\cite[Theorem 4.1]{lenstra}: if $a_n>0$ for all $n$, then $d_{\mathbb{P}}(S_{m, d, r})$ is strictly positive. The idea of Lenstra's proof is that the conditions of $p$ not splitting completely in $L_{\ell}$ are independent of each other and of the condition of $p$ splitting completely in $F_{m, d, r}$, except for some finite set of primes $\ell$. In other words, $L_nF_{m, d, r}$ and $L_{\ell}$ are linearly disjoint extensions for $\ell \nmid n$ and $\ell$ large enough. This allows one to write the density in~\eqref{equ4} as an infinite product
\begin{align}\label{equ4.1}
d_{\mathbb{P}}(S_{m, d, r}) = a_n \prod_{\substack{\ell \nmid n\\\ell \in \mathbb{P}}} \left(1 - \frac{1}{[L_{\ell} : \mathbb{Q}]}\right),
\end{align}
where $n$ is the product of exceptional primes and the terms $1 - \frac{1}{[L_{\ell} : \mathbb{Q}]}$ correspond to the density of primes not splitting completely in $L_{\ell}$ (see~\cite[(5.11)]{lenstra}). Note that the infinite product in equation~\eqref{equ4.1} converges to a strictly positive value, therefore giving both a method for determining whether $d_{\mathbb{P}}(S_{m, d, r}) = 0$ or not and a method for calculating this density.

We are now ready to prove our lemmas.

\begin{proof}[Proof of Lemma~\ref{lem:denPos}]
Consider the case $m = d = 1, r = 0$ above. By~\eqref{equ4.1}, it is sufficient to check that $a_n \neq 0$ for all $n$. Fix $n$ and consider the field
$$F_{1, 1, 0}L_n = \mathbb{Q}(\zeta_{|4abh|}, \zeta_{q(n)}, a^{1/2h}, a^{1/q(n)}).$$
Let $\sigma : F_{1, 1, 0}L_n \to F_{1, 1, 0}L_n$ be the automorphism which fixes $\zeta_{|4abh|}$, $\zeta_{q(n)}$ and $a^{1/2h}$, and which maps $a^{1/q(\ell)}$ to $\zeta_{q(\ell)}a^{1/q(\ell)}$ for all $\ell \mid n$. Assuming that this is well-defined, the result follows from the Chebotarev density theorem (since then $\sigma$ fixes $F_{1,1,0}$ but none of $L_{\ell}$). To prove that there exists such a map $\sigma$, let $K:= \mathbb{Q}(\zeta_{|4abh|}, \zeta_{q(n)}, a^{1/2h})$. It suffices to prove that extending $K$ by the elements $a^{1/q(\ell)}, \ell \mid n$ one by one always increases the degree of $K$ by a factor of $\ell$ (note that the degree increases by at most a factor of $\ell$). In other words, it suffices to prove that the degree of $a^{1/q(\ell)}$ over $K':= \mathbb{Q}(\zeta_{|4abh|}, \zeta_{q(n)}, a^{1/2h}, a^{1/q(n/\ell)})$ is $\ell$. 

To prove this we use the following simple observation (see~\cite[Proposition 3.5]{jarviniemi}): Let $a$ be a positive integer which is not a perfect power. For positive integers $s$ and $t$ with $s \mid t$, the degree of $\mathbb{Q}(\zeta_t, a^{1/s})/\mathbb{Q}(\zeta_t)$ is either $s$ or $s/2$, where the latter case occurs if and only if $\sqrt{a} \in \mathbb{Q}(\zeta_t)$.\footnote{The proof is by basic Galois theory. See~\eqref{equ4.2} below for the proof of a closely related statement.
}
This result will be used subsequently without further mention.

We now have
\begin{align*}
[K'(a^{1/q(\ell)}) : K']&= \frac{[K'(a^{1/q(\ell)}) : \mathbb{Q}(\zeta_{|4abh|}, \zeta_{q(n)})]}{[K' : \mathbb{Q}(\zeta_{|4abh|}, \zeta_{q(n)})]} \\
&= \frac{[\mathbb{Q}(\zeta_{|4abh|}, \zeta_{q(n)}, a^{1/2h}, a^{1/q(n)}) : \mathbb{Q}(\zeta_{|4abh|}, \zeta_{q(n)})]}{[\mathbb{Q}(\zeta_{|4abh|}, \zeta_{q(n)}, a^{1/2h}, a^{1/q(n/\ell)}) : \mathbb{Q}(\zeta_{|4abh|}, \zeta_{q(n)})]},
\end{align*}
where by the above observation the degrees in the numerator and denominator are ``what one would expect'', namely $\lcm(2h, q(n))/2$ and $\lcm(2h, q(n/\ell))/2$, respectively. Here dividing by $2$ accounts for the fact that $\sqrt{a} \in \mathbb{Q}(\zeta_{|4abh|})$. This concludes the proof of Lemma~\ref{lem:denPos}.
\end{proof}

\begin{proof}[Proof of Lemma~\ref{lem:denInv}]
We then prove Lemma~\ref{lem:denInv}.  Fix $m \equiv 0 \pmod{|2ab|}$ and $d$, and let $n$ be the product of exceptional primes over all $1\leq r\leq m$ so that~\eqref{equ4.1} is valid for all $r$. Let $a_n = a_n(r)$ be as in~\eqref{equ4.1}. Recall that $a_n(r)$ is equal to the density of elements $\sigma \in \Gal(F_{m, d, r}L_n/\mathbb{Q})$ fixing $F_{m, d, r}$ but not $L_{\ell}$ for any $\ell \mid n$. One may thus write, by inclusion-exclusion,
\begin{align}
\label{equ4.1.5}
a_n(r) = \frac{1}{[F_{m, d, r}L_n : \mathbb{Q}]} \sum_{k \mid n} \mu(k) [F_{m, d, r}L_k : \mathbb{Q}].
\end{align}
We now turn to proving that, for any fixed $k$, $[F_{m, d, r}L_k : \mathbb{Q}]$ does not depend on $r$. This follows from the following more general claim: Let $m, r, s$ and $t$ be positive integers satisfying $4abh\mid t, m\mid t, s \mid t$ and $2h \mid s$. We have
\begin{equation}
\label{equ4.2}
[\mathbb{Q}(\zeta_t, a^{1/s}, (ba^{-r})^{1/2hm}) : \mathbb{Q}] = \phi(t)\frac{ms}{2}.
\end{equation}
In particular the degrees of the extensions $F_{m, d, r}L_k/\mathbb{Q}$ do not depend on $r$.

Note that the degree in~\eqref{equ4.2} is at most $\phi(t)\frac{ms}{2}$. Indeed,
\begin{align*}
&[\mathbb{Q}(\zeta_t, a^{1/s}, (ba^{-r})^{1/2hm}) : \mathbb{Q}] \\
&= [\mathbb{Q}(\zeta_t) : \mathbb{Q}][\mathbb{Q}(\zeta_t, a^{1/s}) : \mathbb{Q}(\zeta_t)][\mathbb{Q}(\zeta_t, a^{1/s}, (ba^{-r})^{1/2hm}) : \mathbb{Q}(\zeta_t, a^{1/s})].
\end{align*}
The first term equals $\phi(t)$. The second one is at most $\frac{s}{2}$, as $\sqrt{a} \in \mathbb{Q}(\zeta_t)$. The third term is similarly at most $m$, as $(ba^{-r})^{1/2h} \in \mathbb{Q}(\zeta_t, a^{1/s})$, since $|b|^{1/2h}\in \mathbb{Q}(\zeta_{|4abh|})$.

We then prove that this bound is tight. We claim that the numbers
$$a^{e/s}(ba^{-r})^{f/2mh},$$
where $0 \le e < s/2$ and $0 \le f < m$, are linearly independent over $\mathbb{Q}(\zeta_t)$, which gives the lower bound for the degree implicit in~\eqref{equ4.2}. By~\cite{garrett} it suffices to check that
\begin{align}\label{equ30}
a^{e/s}(ba^{-r})^{f/2mh} \not\in \mathbb{Q}(\zeta_t)
\end{align}
for any such $e$ and $f$ not both $0$. Suppose that~\eqref{equ30} fails. Let
$$C := a^{2hem - rfs}b^{sf}.$$
Now $C$ is an integer for which $C^{1/2hms} \in \mathbb{Q}(\zeta_t)$. Thus, arguing as before, $|C|$ must be a perfect $hms$th power of a rational number. Since $a$ and $b$ are coprime, this means that
$$hms \mid 2hem - rfs$$
and
$$hms \mid hsf,$$
where we used the fact that $a$ is not a perfect power and $|b|$ is an $h$th power. The latter condition gives $m\mid f$, and since $0 \le f < m$, we have $f = 0$. Now the first condition gives $s \mid 2e$, so we must have $e = 0$. This concludes the proof of Lemma~\ref{lem:denInv}.
\end{proof}

\section{Applying the second moment method}
\label{sec:chungBound}

As before, let $\ell(p)$ be the smallest positive integer $l$ for which $a^{l}\equiv b\pmod{p}$ (whenever it exists).

Let the set $S$ be as constructed in Section~\ref{sec:equiLog}. For $p\in S$, let 
\begin{align*}
A_p:=\{n\leq x:\,\, n\equiv \ell(p)\pmod{\ord_p(a)}, \,\, n\neq \ell(p)\},    
\end{align*}
and let $\Pr$ denote the uniform probability measure on $[1,x]\cap \mathbb{N}$. Also denote $S_{\leq y}:=S\cap [1,y]$. By Lemma~\ref{lem:2moment}, in order to prove Theorem~\ref{thm:main} it suffices to prove the estimate
\begin{align}\label{equ5}
\frac{\sum_{p, q \in S_{\leq \sqrt{x}}} \Pr(A_p \cap A_q)}{\left(\sum_{p \in S_{\leq \sqrt{x}}} \Pr(A_p)\right)^2}=1+o(1).
\end{align}
By Lemma~\ref{lem:2moment}, the quotient above is always $\geq 1$, so it suffices to establish an upper bound of $\leq 1+o(1)$. Recalling that $\ord_p(a)=\frac{p-1}{2h}$ for $p\in S$, the denominator is
\begin{align}\label{equ6}
\left(\sum_{p\in S_{\leq \sqrt{x}}}\frac{1}{\ord_p(a)}+O\left(\frac{1}{x}\right)\right)^2=(1+o(1))4h^2d_{\mathbb{P}}(S)^2(\log \log x)^2   
\end{align}
by Lemma~\ref{lem:denPos} and partial summation.

The numerator can be written as
\begin{align*}
&\sum_{m \le \sqrt{x}} \sum_{\substack{p, q \in S_{\leq \sqrt{x}} \\ (\ord_p(a), \ord_q(a)) = m \\ m \mid \ell(p) - \ell(q)}} \left(\frac{m}{\ord_p(a)\ord_q(a)}+O\left(\frac{1}{x}\right)\right)\\
&=\sum_{m \le \sqrt{x}} \sum_{\substack{p, q \in S_{\leq \sqrt{x}} \\ (\ord_p(a), \ord_q(a)) = m \\ m \mid \ell(p) - \ell(q)}} \frac{m}{\ord_p(a)\ord_q(a)}+o(1).
\end{align*}
Note that $2ab \mid m$. Let $N$ be a large, fixed integer. By M\"obius inversion, we write
\begin{align*}
& \sum_{m \le \sqrt{x}} \sum_{\substack{p, q \in S_{\leq \sqrt{x}} \\ (\ord_p(a), \ord_q(a)) = m \\ m \mid \ell(p) - \ell(q)}} \frac{m}{\ord_p(a)\ord_q(a)}\\
&\le \sum_{m \le N} \sum_{\substack{p, q \in S_{\leq \sqrt{x}} \\ m\mid \ord_p(a),\ord_q(a)\\ m \mid \ell(p) - \ell(q)}} \frac{m}{\ord_p(a)\ord_q(a)} \sum_{d \mid (\ord_p(a), \ord_q(a))/m} \mu(d) \\
&+\sum_{m > N} \sum_{\substack{p, q \in S_{\leq \sqrt{x}} \\ m \mid \ord_p(a), \ord_q(a) \\ m \mid \ell(p) - \ell(q)}} \frac{m}{\ord_p(a)\ord_q(a)}\\
&= \sum_{m \le N} \sum_{d \le \sqrt{x}} \sum_{\substack{p, q \in S_{\leq \sqrt{x}} \\ md \mid \ord_p(a), \ord_q(a) \\ m \mid \ell(p)-\ell(q)}} \frac{m\mu(d)}{\ord_p(a)\ord_q(a)} \ +  \sum_{m > N} \sum_{\substack{p, q \in S_{\leq \sqrt{x}} \\ m \mid \ord_p(a), \ord_q(a) \\ m \mid \ell(p) - \ell(q)}} \frac{m}{\ord_p(a)\ord_q(a)}.
\end{align*}

The problem thus reduces to considering sums of the form
\begin{align*}
\sum_{\substack{p, q \in S_{\leq \sqrt{x}} \\ md \mid \ord_p(a), \ord_q(a) \\ m \mid \ell(p) - \ell(q)}} \frac{1}{\ord_p(a)\ord_q(a)} &= 
\sum_{r = 0}^{m-1} \Big(\sum_{\substack{p \in S_{\leq \sqrt{x}} \\ md \mid \ord_p(a) \\ \ell(p) \equiv r \pmod{m}}} \frac{2h}{p-1}\Big)^2\\
&=
4h^2 \sum_{r = 0}^{m-1} \Big(\sum_{\substack{p \le \sqrt{x}\\ p \in S_{m, d, r}}} \frac{1}{p-1}\Big)^2
\end{align*}
for various $d, m\leq \sqrt{x}$.

Define
$$\Sigma_{m, d, r} := \Big(\sum_{\substack{p \le \sqrt{x}\\ p \in S_{m, d, r}}} \frac{1}{p-1}\Big)^2$$
and
$$\Sigma_{m, d}:= \sum_{r = 0}^{m-1} \Sigma_{m, d, r}.$$
Now we have
$$\sum_{p, q \le x} \Pr(A_p \cap A_q) \le 4h^2 \left(\sum_{m \le N} \sum_{d \le \sqrt{x}} m\mu(d)\Sigma_{m, d} + \sum_{m > N} m\Sigma_{m, 1}\right).$$
In Sections~\ref{sec:gcdSmall} and~\ref{sec:gcdLarge}, we will prove that for any $N\geq 1$ and for $x\geq x_0(N)$ (where $x_0(N)$ is a fast enough growing function of $N$) we have
\begin{align}\label{equ7}
\sum_{m \le N} \sum_{d \le \sqrt{x}} m\mu(d)\Sigma_{m, d} = (1 + O(1/N))d_{\mathbb{P}}(S)^2(\log \log x)^2
\end{align}
and
\begin{align}\label{equ8}
\sum_{m > N} m\Sigma_{m, 1} \ll (\log \log x)^2/N
\end{align}
Combining~\eqref{equ6},~\eqref{equ7} and~\eqref{equ8} and letting $x \to \infty$ gives a bound of $\le 1 + O(1/N)$ for the right-hand side of~\eqref{equ5}, and letting $N\to \infty$ slowly in terms of $x$ gives Theorem~\ref{thm:main}. Thus, our remaining task in the proof of Theorem~\ref{thm:main} is proving~\eqref{equ7} and~\eqref{equ8}.

\section{Contribution of small gcd}
\label{sec:gcdSmall}

In this section, we prove the desired estimate~\eqref{equ7}. We handle the cases $md \le N^3$ and $md > N^3$ separately.

\subsection{Case 1. \texorpdfstring{$md \le N^3$}{md<N3}}

Let $\pi_{m, d, r}(x):= \sum_{p \le x, p \in S_{m, d, r}} 1$. By partial summation, we have
\begin{align*}
\Sigma_{m, d, r}^{1/2} &= \sum_{\substack{p \le \sqrt{x}\\ p \in S_{m, d, r}}} \frac{1}{p-1}\\
&=\frac{\pi_{m,d,r}(\sqrt{x})}{\sqrt{x}-1}+\int_{2}^{\sqrt{x}}\frac{\pi_{m,d,r}(t)}{(t-1)^2}\, dt.
\end{align*}
The first term here is $O(1/\log x)$. The second term can be bounded by splitting the integration interval into parts $[2,y_0(N)]$ and $[y_0(N),\sqrt{x}]$, where $y_0(N)$ is a fast growing function of $N$ such that
\begin{align*}
\pi_{m,d,r}(t)=(1+O(N^{-5}))d_{\mathbb{P}}(S_{m,d,r})\frac{t}{\log t} \quad \textnormal{for}\quad t\geq y_0(N),   
\end{align*}
uniformly for $m,d,r\leq N$; such a function exists by Lemma~\ref{lem:denInv}. We then have
\begin{align*}
\int_{2}^{y_0(N)}\frac{\pi_{m,d,r}(t)}{(t-1)^2}\, dt\ll \log \log y_0(N)    
\end{align*}
and 
\begin{align*}
\int_{y_0(N)}^{\sqrt{x}}\frac{\pi_{m,d,r}(t)}{(t-1)^2}\, dt&=\int_{y_0(N)}^{\sqrt{x}}\frac{(1+O(N^{-5}))d_{\mathbb{P}}(S_{m,d,r})t/\log t}{(t-1)^2}\, dt\\
&=(1+O(N^{-5}))d_{\mathbb{P}}(S_{m,d,r})\log \log x+O(1).    
\end{align*}

All in all, we have
$$\Sigma_{m, d, r}^{1/2} = (1+O(N^{-5}))d_{\mathbb{P}}(S_{m, d, r})\log \log x + O(\log \log y_0(N)).$$
Squaring gives
$$\Sigma_{m, d, r} = (1+O(N^{-5}))d_{\mathbb{P}}(S_{m, d, r})^2(\log \log x)^2 + O((\log \log x)(\log \log y_0(N))).$$
Summing over $r$ and using Lemma~\ref{lem:denInv} (which says that $d_{\mathbb{P}}(S_{m,d,r})$ is independent of $r$), we get
\begin{align}\label{equ8.1}
\Sigma_{m, d} = \sum_{r = 0}^{m-1} \Sigma_{m, d, r} = (1+O(N^{-5}))\frac{d_{\mathbb{P}}(S_{1,md,0})^2}{m} (\log \log x)^2 + O\left(m(\log \log x)^{1.1}\right).
\end{align}
We multiply this by $m\mu(d)$ and sum over $m$ and $d$ with $md \le N^3$. For $x$ large enough in terms of $N$, this gives
\begin{align*}
&\sum_{m \le N} \sum_{d \le \frac{N^3}{m}} m\mu(d)\Sigma_{m, d}  \\
&=\sum_{m \le N} \sum_{d \le \frac{N^3}{m}} \mu(d)(1 + O(N^{-5}))d_{\mathbb{P}}(S_{1,md,0})^2(\log \log x)^2+O(N^{3.1}(\log \log x)^{1.1}) \\
&= (\log \log x)^2\sum_{m \le N} \sum_{d \le \frac{N^3}{m}} \mu(d)d_{\mathbb{P}}(S_{1, md, 0})^2 + O\left(d_{\mathbb{P}}(S)^2N^{-1.9}\right) + O\left(N^{3.1}(\log \log x)^{1.1}\right) \\
&= (1 + O(N^{-1}))d_{\mathbb{P}}(S_1)^2(\log \log x)^2,
\end{align*}
where we used 
$$\sum_{dm\leq N^3}\mu(d)a_{dm}=\sum_{k\leq N^3}a_k\sum_{d\mid k}\mu(d)=a_1.$$

\subsection{Case 2. \texorpdfstring{$md > N^3$}{md>N3}}

Now we bound the contribution to~\eqref{equ7} from the terms $md>N^3$. Crude estimation gives
\begin{align*}
\sum_{r = 0}^{m-1} \Sigma_{m, d, r} &=
\sum_{r = 0}^{m-1} \Big(\sum_{\substack{p \le \sqrt{x}\\ p\in S_{m, d, r}}} \frac{1}{p-1} \Big)^2\\
&\leq \Big(\sum_{p \le x, p \equiv 1 \pmod{md}} \frac{1}{p-1}\Big)^2.
\end{align*}
We split the summation over $p$ into the intervals $[md,(md)^2]$ and $((md)^2,x]$. Applying the Brun--Titchmarsh inequality for $p\in [(md)^2,x]$ and a trivial estimate for $p\in [md,(md)^2]$ gives
\begin{align*}
\Big(\sum_{\substack{p \le x\\ p \equiv 1 \pmod{md}}} \frac{1}{p-1}\Big)^2 \ll \frac{(\log \log x)^2}{\phi(md)^2}+\Big(\frac{\log(md)}{md}\Big)^2.
\end{align*}
Since $m\leq N, md>N^3$ implies $m\leq (md)^{1/3}$, we can bound the contribution to~\eqref{equ7} from the terms $md>N^3$ by
\begin{align*}
&\sum_{\substack{md > N^3\\m\leq N}}m \left(\frac{(\log \log x)^2}{\phi(md)^2}+\frac{\log^2(md)}{(md)^2}\right) \\
&\ll (\log \log x)^2\sum_{k > N^3}k^{1/3}\left(\frac{\tau(k)}{\phi(k)^2}+\frac{\tau(k)\log^2 k}{k^2}\right)\\
&\ll (\log \log x)^2/N,
\end{align*}
which is the desired bound.

\section{Contribution of large gcd}
\label{sec:gcdLarge}

We are now left with the task of proving~\eqref{equ8}. Define 
\begin{align*}
S_{m,d,r}':=\{p\in \mathbb{P}:\,\, p \equiv 1 \pmod{2hmd},\,\, ba^{-r}\equiv z^{2hm}\pmod{p}\,\, \textnormal{for some}\,\,\, z\}.    
\end{align*}
Furthermore, let $T_m$ be the set of primes $p$ for which $p \equiv 1 \pmod{2hm}$ and for which $a$ is not a $2hq$th power $\pmod{p}$ for any prime $q \mid m$. We now see that $S_{m, d, r}\subset S_{m, d, r}' \cap T_m$, where these larger sets $S_{m, d, r}'$ and $T_m$ are easier to control.

We estimate
\begin{align*}
\sum_{r = 0}^{m-1} \Sigma_{m, 1, r} &\ll
\sum_{r = 0}^{m-1} \Big(\sum_{\substack{p \le \sqrt{x}\\p \in S_{m, 1, r}' \cap T_m}} \frac{1}{p} \Big)^2.    
\end{align*}
Call the expression inside the $r$ sum $\Sigma_{m, r}'$. We divide the sum $\Sigma_{m, r}'$ into three parts depending on the relative size of $p$ and $m$ as follows. Set $B = 10$. We have
\begin{align*}
\Sigma_{m, r}' = \Big(\sum_{\substack{p \le \sqrt{x}\\ p \in S_{m, 1,r}' \cap T_m}}  \frac{1}{p}\Big)^2\ll \Big(\sum_{\substack{m\leq p \le mF(m)\\ p \in S_{m, 1,r}' \cap T_m}} \frac{1}{p}\Big)^2 + \Big(\sum_{\substack{mF(m)\leq p \le m^B\\ p \in S_{m, 1,r}' \cap T_m}} \frac{1}{p}\Big)^2+ \Big(\sum_{\substack{m^B \le p \le \sqrt{x}\\ p \in S_{m, 1,r}'}} \frac{1}{p}\Big)^2,
\end{align*}
where
\begin{align*}
F(m):=\exp((\log m)^{1/2}).    
\end{align*}
Denote the sums on the right by $\Sigma_{m, r, 1}'$, $\Sigma_{m, r, 2}'$, $\Sigma_{m, r, 3}'$, respectively. In the following subsections we prove that the sum of $\Sigma_{m, r, i}'$ over all $m$ and $r$ is small enough for each $i\in \{1,2,3\}$.

\subsection{Bounds for very large gcd}
\label{sec:primeSmall}

We begin by estimating $\Sigma_{m,r,1}'$ on average over $r$ and $m$. This sum will be dealt with using Selberg's sieve, and in particular this part of the proof is unconditional. We have
\begin{align}\label{e5}\begin{split}
\sum_{r=0}^{m-1}\Big(\sum_{\substack{m\leq p\leq mF(m)\\ p \in S_{m, 1, r}' \cap T_m}}\frac{1}{p}\Big)^2& \ll \Big(\sum_{\substack{m \le p \le mF(m) \\ p \equiv 1 \pmod{m}}} \frac{1}{p}\Big)^2 \\
&= \sum_{\substack{m\leq p,q\leq mF(m)\\p\equiv q\equiv 1 \pmod{m}}}\frac{1}{pq}.
\end{split}
\end{align}

We claim that for any $M\gg 1$, for $m\in [M,2M]$ and $M\leq P\leq Q\leq M^{2}$, we have
\begin{align}\label{e111}
\sum_{\substack{P\leq p\leq 2P\\Q\leq q\leq 2Q\\p\equiv q\equiv 1 \pmod{m}}}\frac{1}{pq}\ll \frac{1}{M(\log M)^2}+\frac{1_{P/2\leq Q\leq 2P}}{Q\log Q}.  
\end{align}
Once this has been shown, multiplying by $m$ and summing over $P$ and $Q$ dyadically  we can upper bound the contribution of $M>N$ by
\begin{align*}
\sum_{m > N} m\sum_{r = 0}^{m-1} \Sigma_{m, r, 1}'\ll  \sum_{\substack{N<M\leq \sqrt{x}\\M=2^{j}}}\left(\frac{(\log F(2M))^2}{(\log M)^2}+\frac{1}{\log M}\right)
\ll \sum_{\substack{N<M\leq \sqrt{x}\\M=2^{j}}}\frac{1}{\log M}\ll \log \log x
\end{align*}
by our choice $F(m)=\exp((\log m)^{1/2})$. This is certainly $o((\log \log x)^2)$.

To prove~\eqref{e111}, we first note that the contribution of the terms $p=q$ to that sum is $\ll 1/(Q\log Q)$, so it suffices to prove a bound of $\ll 1/(M(\log M)^2)$ for the off-diagonal contribution. Also observe that if $p\equiv 1\pmod m$, then $p=am+1$ for some $P/M\leq a\leq 2P/M$. Thus, the contribution of $p\neq q$ to~\eqref{e111} is
\begin{align*}
\ll  \frac{1}{PQ}\sum_{\substack{P/M\leq a\leq 2P/M\\Q/M\leq b\leq 2Q/M,\,b\neq a}}|\{m\leq 2M:\,\, am+1,bm+1\in \mathbb{P}\}|.   
\end{align*}
By Selberg's sieve, in the form of~\cite[Theorem 6.4]{iwaniec-kowalski} (see also formulas (6.83)--(6.86) there), this is further bounded by
\begin{align}\label{e112}
\ll \frac{1}{PQ}\sum_{\substack{P/M\leq a\leq 2P/M\\Q/M\leq b\leq 2Q/M,\,b\neq a}}\frac{M}{(\log M)^2}\prod_{p\mid ab(a-b)}\left(1+\frac{1}{p}\right).     
\end{align}
Note that by Cauchy--Schwarz for any $0\leq h\leq y$ we have  
\begin{align}\label{e116}
\sum_{n\leq y}\frac{n}{\phi(n)}\cdot \frac{n+h}{\varphi(n+h)}\leq\left(\sum_{n\leq y}\left(\frac{n}{\phi(n)}\right)^2\right)^{1/2}\left(\sum_{n\leq y}\left(\frac{n+h}{\phi(n+h)}\right)^2\right)^{1/2}\ll y.     
\end{align}
Therefore,~\eqref{e112} is at most
\begin{align*}
\ll \frac{1}{PQ}\cdot \frac{P}{M}\cdot \frac{Q}{M}\cdot \frac{M}{(\log M)^2}\ll \frac{1}{M(\log M)^2},    
\end{align*}
as wanted.

\subsection{Bounds for medium gcd}
\label{sec:primeLarge}

We then deal with the sums $\Sigma_{m,r,3}'$ for which we assume GRH (but not Hypothesis~\ref{hyp}). We divide $[m^B, \sqrt{x}]$ into dyadic intervals $[y/2, y]$ to obtain
\begin{align}\label{equ31}
\Sigma_{m, r, 3}' &\ll \Big(\sum_{\substack{m^{B}\leq y\leq \sqrt{x}\\y=2^j}}\frac{1}{y} \sum_{\substack{p \in [y/2, y]\\ p \in S_{m, 1, r}'}} 1\Big)^2.
\end{align}
Note then that $p\in S_{m,1,r}'$ implies $p\equiv 1\pmod m$ and that $ba^{-r}$ is an $m$th power $\pmod p$, so by Galois theory $p$ splits completely in $L:=\mathbb{Q}(\zeta_m,(ba^{-r})^{1/m})$. Thus the Artin symbol $\left(\frac{L/\mathbb{Q}}{p}\right)$ is equal to $1$. 

Now we are in a position to apply the Chebotarev density theorem (Lemma~\ref{lem:chebotarev}). Under GRH, it gives
\begin{align*}
\pi_C(y)=\frac{|C|}{[L:\mathbb{Q}]}\textnormal{Li}(y)+O\left(\frac{|C|}{[L:\mathbb{Q}]}\sqrt{y}(\log \textnormal{disc}(L/\mathbb{Q})+[L:\mathbb{Q}]\log y)\right),    
\end{align*}
where $C$ is any conjugacy class of $\Gal(L/\mathbb{Q})$. In the case of the specific extension $L$ above, we easily see that  $[L:\mathbb{Q}]\gg m\phi(m)$ (see Lemma~\ref{lem:indBound}), and also by~\cite[Lemma 1.2]{cooke} we have $\log \textnormal{disc}(L/\mathbb{Q})\ll m^2\log m$, so if $C$ is the identity of $\Gal(L/\mathbb{Q})$, we get
\begin{align*}
\pi_C(y)\ll \frac{1}{m\phi(m)}\textnormal{Li}(y)\quad \textnormal{for}\quad y\geq m^{4.1}.    
\end{align*}

Therefore,~\eqref{equ31} is
\begin{align*}
&\ll \Big(\sum_{\substack{m^B\leq y\leq \sqrt{x}\\y=2^j}} \frac{1}{y} \cdot \frac{y}{(\log y)m \phi(m)}\Big)^2\ll \frac{(\log \log x)^2}{(m\phi(m))^2}.
\end{align*}
Summing this over all $0\leq r\leq m-1$, multiplying by $m$ and summing over $m > N$ gives
$$\sum_{m > N} m\sum_{r = 0}^{m-1} \Sigma_{m, r, 3}' \ll \sum_{m > N}\frac{(\log \log x)^2}{\phi(m)^2}\ll \frac{(\log \log x)^2}{N}.$$
This is a good enough estimate.

\subsection{Bounds for large gcd}
\label{sec:primeMedium}

We then turn to estimating $\Sigma_{m, r, 2}'$ on average over $r$ and $m$. This is the only part of our proof where Hypothesis~\ref{hyp} is used.

Let
\begin{align}\label{e103}
\mathcal{M}:=\{N\leq m\leq \sqrt{x}:\,\, \exists\,\, m'\mid m: m'\in [(\log m)^{\epsilon'},m^{\epsilon}]\cap \mathbb{P}\},
\end{align}
where $\epsilon > 0$ is small but fixed and $\epsilon'=\epsilon'(N)$ is chosen to be small enough in terms of $N$. For each $m \in \mathcal{M}$ we let $m'$ denote such a divisor of $m$.

We wish to estimate the quantity
\begin{align}\label{e90}
\Sigma_{m,r,2}'&= \Big(\sum_{\substack{mF(m)\leq p\leq m^B\\p\in S'_{m,1,r}\cap T_m}}\frac{1}{p}\Big)^2.
\end{align}

We split
\begin{align}\label{e91}
&\sum_{M\leq m\leq 2M}\sum_{r=0}^{m-1} \Sigma_{m,r,2}'\nonumber\\
&=  \sum_{\substack{M\leq m\leq 2M\\m\in \mathcal{M}}}\sum_{r=0}^{m-1}\Big(\sum_{\substack{mF(m)\leq p\leq m^{B}\\p\in S_{m, 1, r}' \cap T_m}}\frac{1}{p}\Big)^2+\sum_{\substack{M\leq m\leq 2M\\m\not \in \mathcal{M}}}\sum_{r=0}^{m-1}\Big(\sum_{\substack{mF(m)\leq p\leq m^{B}\\p\in S_{m, 1, r}' \cap T_m}}\frac{1}{p}\Big)^2, 
\end{align}

We now present a general bound for the contribution of a single $m$ to the first sum in~\eqref{e91}. This is the only part of the paper where Hypothesis~\ref{hyp} is put into use.

\begin{lemma}\label{lem:hypothesis}
Assume Hypothesis~\ref{hyp}. Let $P\geq 2$, and let $m\in [M,2M]$ and $m'\in \mathbb{P}$ with $m'\mid m$ and $m'\in [(\log m)^{\epsilon'},m^{\epsilon}]$. Then, provided that $P\geq mF(m)$, we have 
\begin{align*}
\sum_{r=0}^{m-1}\Big(\sum_{\substack{P\leq p\leq 2P\\p\in S'_{m,1,r}\cap T_m}}1\Big)^2\ll_{\epsilon',\epsilon}  \frac{P^2}{\phi(m)^2(\log P)^{2}}\max\left\{\frac{1}{(m')^{\epsilon}},\frac{1}{(\log P)^2}\right\}. 
\end{align*}
\end{lemma}

\begin{proof}
Note first that if $\ell(p) \equiv r \pmod{m}$, then $\ell(p) \equiv r \pmod{m'}$. Observe also that 
\begin{align}
\label{equ9.5}
S_{m, 1, r}' \cap S_{m, 1, r'}' \cap T_m = \emptyset
\end{align}
for any $r \neq r'$: if $p \in S_{m, 1, r}' \cap S_{m, 1, r'}'$, then $a^{r - r'}$ is a perfect $2hm$th power and thus $a$ must be a $2hq$th power for some prime $q \mid m$, implying $p \not\in T_m$. Also note the trivial inequality \begin{align}\label{e102}
x_1^2+\cdots+x_k^2\leq (x_1+\cdots +x_k)^2,\quad  x_i\geq 0.    
\end{align}
Using~\eqref{equ9.5} and~\eqref{e102}, we obtain
\begin{align*}
\sum_{r=0}^{m-1}\Big(\sum_{\substack{P\leq p\leq 2P\\p\in S'_{m,1,r}\cap T_m}}1\Big)^2  \ll \sum_{r=0}^{m'-1} \Big(\sum_{\substack{P \le p \le 2P \\ p \in T_{m, m', r}}} 1\Big)^2,
\end{align*}
where $T_{m, m', r}$ is the set of primes $p$ which split completely in $\mathbb{Q}(\zeta_{2hm}, a^{1/2h}, (ba^{-r})^{1/2hm'})$ but which do not split completely in $\mathbb{Q}(\zeta_{2hq}, a^{1/2hq})$ for any $q \mid m$. In particular,
\begin{align*}
T_{m,m',r}\subset T'_{m,m',r}:=\{p:\,\, p\,\, \textnormal{splits completely in}\,\, \mathbb{Q}(\zeta_m,(ba^{-r})^{1/m'})\}.    
\end{align*}
Hence, under Hypothesis~\ref{hyp}, we have
\begin{align*}
\sum_{r=0}^{m'-1}\Big(\sum_{\substack{P\leq p\leq 2P\\ p \in T_{m, m', r}}} 1\Big)^2
&\ll_{\epsilon',\epsilon} \frac{P^2}{\phi(m)^2(\log P)^{2}}\max\left\{\frac{1}{(m')^{\epsilon}},\frac{1}{(\log P)^2}\right\},
\end{align*}
which produces the assertion of the lemma.
\end{proof}

We apply Lemma~\ref{lem:hypothesis} to the sums on the right of~\eqref{e91}. Denote
\begin{align*}
D_{\leq y}(m):=\max_{\substack{d\in [1,y]\cap \mathbb{P}\\d\mid m}}d .   
\end{align*}

Splitting dyadically over $p$ and $m$, we see that
\begin{align}\label{e113}
&\sum_{\substack{M\leq m\leq 2M\\m\in \mathcal{M}}}\sum_{r=0}^{m-1}\Big(\sum_{\substack{mF(m) \le p \le m^B \\ p \in S'_{m, 1, r}\cap T_m}} \frac{1}{p}\Big)^2\nonumber\\
&\ll \sum_{\substack{M\leq m\leq 2M\\m\in \mathcal{M}}}\sum_{\substack{MF(M)\leq P,Q\leq (2M)^{B}\\P=2^{j_1},Q=2^{j_2}}}\sum_{r=0}^{m'-1}\Big(\sum_{\substack{P\leq p\leq 2P\\p\in \mathcal{T}_{m,m',r}}}\frac{1}{P}\Big)\Big(\sum_{\substack{Q\leq q\leq 2Q\\q\in \mathcal{T}_{m,m',r}}}\frac{1}{Q}\Big)\nonumber\\
&\ll \sum_{\substack{M\leq m\leq 2M\\m\in \mathcal{M}}}\sum_{\substack{MF(M)\leq P,Q\leq (2M)^{B}\\P=2^{j_1},Q=2^{j_2}}}\sum_{r=0}^{m'-1}\Big(\sum_{\substack{P\leq p\leq 2P\\p\in \mathcal{T}_{m,m',r}}}\frac{1}{P}\Big)^2\nonumber\\
&\ll\sum_{\substack{M\leq m\leq 2M\\m\in \mathcal{M}}}\left(\frac{1}{\phi(m)^2(\log M)^2}+\frac{1}{\phi(m)^2D_{\leq m^{\epsilon}}(m)^{\epsilon}}\right)\nonumber\\
&\ll_{\epsilon', \epsilon} \frac{1}{M(\log M)^{2}}+\sum_{\substack{M\leq m\leq 2M\\m\in \mathcal{M}}}\frac{1}{\phi(m)^2D_{\leq m^{\epsilon}}(m)^{\epsilon}}.
\end{align}  

Regarding the first term in~\eqref{e113}, if we multiply it by $M$ and summing over $M\leq \sqrt{x}$ that are powers of two, we get a bound of
\begin{align*}
\ll \sum_{\substack{M\leq \sqrt{x}\\M=2^j}}\frac{1}{(\log M)^{2}}\ll 1,   
\end{align*}
which is certainly small enough.

Now, to conclude the proof of Theorem~\ref{thm:main}, it suffices to prove the following two claims.

\begin{lemma}\label{lem:sparse} We have 
\begin{align}\label{e96}
\sum_{\substack{m\leq \sqrt{x}\\m\not \in \mathcal{M}}}m\sum_{r=0}^{m-1}\Big(\sum_{\substack{mF(m) \le p \le m^B \\ p \in S_{m, 1, r}\cap T_m}} \frac{1}{p}\Big)^2=o((\log \log x)^2).
\end{align}
\end{lemma}

\begin{lemma}\label{lem:sum}
We have 
\begin{align}\label{e106}
\sum_{\substack{M\leq m\leq 2M\\m\in \mathcal{M}}}\frac{1}{\phi(m)^2D_{\leq m^{\epsilon}}(m)^{\epsilon}}\ll_{\epsilon}  \frac{1}{M(\log M)^{1+\epsilon\cdot \epsilon'/2}}.   
\end{align}
\end{lemma}

Note that we have
\begin{align*}
\sum_{M=2^j}\frac{1}{(\log M)^{1+\epsilon\cdot \epsilon'/2}}\ll_{\epsilon',\epsilon} 1,
\end{align*}
so the sum over $M$ of the contributions given by Lemma~\ref{lem:sum} is small.

\begin{proof}[Proof of Lemma~\ref{lem:sparse}]
Analogously to~\eqref{e5}, we first make for $m\in [M,2M]$ the crude estimate
\begin{align*}
\Big(\sum_{\substack{mF(m) \le p \le m^B \\ p \in S'_{m, 1, r}\cap T_m}} \frac{1}{p}\Big)^2\leq \Big(\sum_{\substack{mF(m)\leq p\leq m^B\\p\equiv 1\pmod m}}\frac{1}{p}\Big)^2\leq \sum_{MF(M)\leq p,q\leq (2M)^{B}}\frac{1_{p\equiv q\equiv 1\pmod m}}{pq}.    
\end{align*}
Note that applying the Brun--Titchmarsh inequality to this would cost us some nontrivial factors, since that inequality involves savings of $1/(\log(P/m))$ rather than $1/\log P$ if $p\in [P,2P]$. Therefore, we instead exploit the summation over $m$. Let us restrict to $m\in [M,2M]$, $p\in [P,2P]$ and $q\in [Q,2Q]$; we will later sum dyadically over different $M, P, Q$. Let
\begin{align*}
\mathcal{R}=\mathcal{R}_M:=\{m\in [M,2M]: \,\, p\nmid m\,\, \textnormal{for all}\,\, (\log 2M)^{\epsilon'}\leq p\leq M^{\epsilon}\},    
\end{align*}
and observe that $[M,2M]\setminus \mathcal{M}\subset \mathcal{R}$. Our task is thus to upper bound
\begin{align}\label{e5.5}
\frac{M}{PQ}\sum_{\substack{P/M\leq n\leq 2P/M\\Q/M\leq n'\leq 2Q/M}} \sum_{\substack{M\leq m\leq 2M\\m\in \mathcal{R}}} 1_{nm+1\in \mathbb{P}}1_{n'm+1\in \mathbb{P}}.  
\end{align}

Let us first estimate $|\mathcal{R}|$. Let $\mathcal{P}(z):=\prod_{p\leq z}p$. Note that
\begin{align*}
|\mathcal{R}|\leq \sum_{\textnormal{rad}(d)\mid \mathcal{P}((\log (2M))^{\epsilon'})}\sum_{\substack{r\leq 2M/d\\(r,\mathcal{P}(M^{\epsilon}))=1}}1:=S_1+S_2,    
\end{align*}
where $S_1$ is the part of the sum with $d\leq M^{1/2}$ and $S_2$ is the complementary part. By Mertens' theorem, we have
\begin{align}\label{e115}\begin{split}
S_1&\ll_{\epsilon}\sum_{\textnormal{rad}(d)\mid \mathcal{P}((\log (2M))^{\epsilon'})}\frac{M}{d\log M} \\
&\ll \frac{M}{\log M}\prod_{p\leq (\log(2M))^{\epsilon'}}\left(1+\frac{1}{p}+\frac{1}{p^2}+\cdots\right)\\
&\ll \epsilon'\frac{M\log \log M}{\log M}.
\end{split}
\end{align}
 When it comes to $S_2$, we crudely estimate
\begin{align*}
S_2&\ll  \sum_{\substack{\textnormal{rad}(d)\mid \mathcal{P}((\log (2M))^{\epsilon'})\\d>M^{1/2}}} \frac{M}{d}\ll M^{0.95}  \sum_{\textnormal{rad}(d)\mid \mathcal{P}((\log (2M))^{\epsilon'})}\frac{1}{d^{0.9}}\\
&\ll M^{0.95}\prod_{p\leq (\log(2M))^{\epsilon'}}\left(1+\frac{1}{p^{0.9}}+\frac{1}{p^{1.8}}+\cdots\right)\\
&\ll M^{0.96}.
\end{align*}
Hence,
\begin{align}\label{e8}
|\mathcal{R}|\ll_{\epsilon} \epsilon'\frac{M\log \log M}{\log M}.    
\end{align}

Next, note that by~\eqref{e8} the diagonal contribution $n=n'$ to~\eqref{e5.5} (crudely dropping the condition that $nm+1\in \mathbb{P}$) is
\begin{align*}
\ll \frac{M}{PQ}\cdot \frac{P}{M}|\mathcal{R}_M|\ll_{\epsilon} \epsilon'\frac{\log \log M}{Q\log M}     
\end{align*}
and it only appears if $P\asymp Q$. Multiplying by $M$ and summing dyadically over $Q$ and $M$,  we get a contribution of $\ll \epsilon'(\log \log x)^2$ for the diagonal terms. Since $\epsilon'=\epsilon'(N)$ for sufficiently fast decaying $\epsilon'(N)$, this is $\ll (\log \log x)^2/N$. We may henceforth restrict to $n\neq n'$ in~\eqref{e5.5}.

We are then left with upper bounding the non-diagonal contribution $n\neq n'$ in~\eqref{e5.5}. To this end, we will apply Selberg's sieve to the set
\begin{align*}
\mathcal{A}=\mathcal{A}^{(n,n')}:=\{(nm+1)(n'm+1):\,\, m\in [M,2M]\cap \mathcal{R}\}.    
\end{align*}
We claim that if $d$ is squarefree, $(d,nn'(n-n'))=1$ and $\mathcal{A}_d:=\{m\in \mathcal{A}:\,\, m\equiv 0\pmod d\}$, then
\begin{align}\label{e6}
|\mathcal{A}_d|=\frac{g(d)}{d}|\mathcal{A}|+R_d,
\end{align}
where $R_d$ satisfies a level of distribution estimate
\begin{align}\label{e7}
\sum_{\substack{d\leq M^{0.1}\\\mu^2(d)=1\\(d,nn'(n-n'))=1}}|R_d|\ll_A M/(\log M)^A,    
\end{align}
and $g:\mathbb{N}\to \mathbb{R}_{\geq 0}$ is the completely multiplicative function given at primes by
\begin{align*}
g(p):=\begin{cases}2,\quad p\leq (\log(2M))^{\epsilon'},\\\frac{2p}{p-1},\quad p>(\log(2M))^{\epsilon'}.
\end{cases}    
\end{align*}
We note that for $m\in \mathcal{R}$ the condition $(nm+1)(n'm+1)\in \mathcal{A}_p$ is equivalent to $m\equiv -1/n\pmod p$ or $m\equiv -1/n'\pmod p$, and the residue classes $-1/n\pmod p$ and $-1/n'\pmod p$ are distinct from each other and from $0 \pmod p$ whenever $(p,nn'(n-n'))=1$. Therefore, in order to prove~\eqref{e7} it suffices to prove that if $\mathcal{R}_{d,a}:=\{m\in \mathcal{R}:m\equiv a\pmod d\}$, then for $a$ coprime to $d$ we have
\begin{align*}
|\mathcal{R}_{d,a}|=\frac{g'(d)}{d}|\mathcal{R}|+R_{d,a}',
\end{align*}
where 
\begin{align}\label{e9}
\sum_{\substack{d\leq M^{0.1}\\\mu^2(d)=1\\(d,nn'(n-n'))=1}}\tau(d)\max_{(a,d)=1}|R_{d,a}'|\ll_A M/(\log M)^A,   
\end{align}
and $g':\mathbb{N}\to \mathbb{R}_{\geq 0}$ is the completely multiplicative function given by
\begin{align*}
g'(p):=\begin{cases}1,\quad p\leq (\log(2M))^{\epsilon'},\\\frac{p}{p-1},\quad p>(\log(2M))^{\epsilon'}.
\end{cases}    
\end{align*}

For proving~\eqref{e9}, we wish to split $R'_{d,a}$ as a bilinear sum and then apply a general Bombieri--Vinogradov type result for such objects.

Observe the decomposition
\begin{align*}
1_{\mathcal{R}}(n)=\alpha*\beta(n)=\alpha_1*\beta(n)+\alpha_2*\beta(n)+\alpha_3*\beta(n),    
\end{align*}
where 
\begin{align*}
&\alpha(n)=1_{n\mid \mathcal{P}((\log M)^{\epsilon'})},\quad \quad \quad \quad\quad \quad \quad \beta(n)=1_{(n,\mathcal{P}(M^{\epsilon}))=1}\\
&\alpha_1(n)=1_{n\mid \mathcal{P}((\log M)^{\epsilon'}),n\leq (\log M)^{10A}},\quad \alpha_2(n)= 1_{n\mid \mathcal{P}((\log M)^{\epsilon'}),\,\,(\log M)^{10A}<n\leq M^{0.1}},\\
&\alpha_3(n)= 1_{n\mid \mathcal{P}((\log M)^{\epsilon'}),\,\,n> M^{0.1}}
\end{align*}
The contribution of $\alpha_3*\beta(n)$ to~\eqref{e9} is negligible (similarly as in~\eqref{e115}), whereas the part $\alpha_2*\beta(n)$ has level of distribution $1/2$ by a variant of the Bombieri--Vinogradov theorem (see~\cite[Theorem 17.4]{iwaniec-kowalski}, where one just needs to verify the Siegel--Walfisz condition for $\beta$. This condition can be proved for $1_{(n,\mathcal{P}(M^{\epsilon}))=1}$ similarly as for the indicator of the primes). We claim that also the part $\alpha_1*\beta$ has level of distribution $1/2$. Since $\alpha_1$ is supported on $[1,(\log M)^{10A}]$, we see that if $\beta$ has level of distribution $1/2$, so does $\alpha_1*\beta$. For the sequence $\beta$, we indeed obtain level of distribution $1/2$ by noting that $\beta(n)=1_{\mathbb{P}}(n)+\sum_{2\leq j\leq 1/\epsilon}1_{n=p_1\cdots p_j,p_i>M^{\epsilon}}$, and here $1_{\mathbb{P}}(n)$ has level of distribution $1/2$ by the classical Bombieri--Vinogradov theorem, whereas each of the terms $1_{n=p_1\cdots p_j,p_i>M^{\epsilon}}$ has level of distribution $1/2$ by~\cite[Theorem 17.4]{iwaniec-kowalski}.

Now we have
\begin{align*}
\sum_{\substack{d\leq M^{1/2-\epsilon}\\\mu^2(d)=1}}\max_{(a,d)=1}|R_{d,a}'|\ll_{A,\epsilon} M/(\log M)^A,       
\end{align*}
and~\eqref{e9} follows from this by Cauchy--Schwarz and the divisor sum estimate $\sum_{d\leq M}\tau(d)^2\ll M(\log M)^{3}$ and the trivial bound $|R_{d,a}'|\ll M/d$.

Taking~\eqref{e6},\eqref{e7},~\eqref{e8} and~\eqref{e116} into account, by Selberg's sieve we obtain for~\eqref{e5.5} the upper bound
\begin{align*}
&\ll \frac{M}{PQ}\cdot \sum_{\substack{P/M\leq n\leq 2P/M\\Q/M\leq n'\leq 2Q/M\\n\neq n'}} \prod_{p\mid nn'(n-n')}\left(1+\frac{2}{p}\right)\frac{1}{(\log P)(\log Q)}\cdot \epsilon'\frac{M\log \log M}{\log M}\\
&\ll \epsilon'\frac{\log \log M}{\log M}\cdot \frac{1}{(\log P)(\log Q)}.
\end{align*}
Summing dyadically over $P,Q\in [MF(M),(2M)^{B}]$ and $M\in [N,\sqrt{x}]$, this becomes
\begin{align*}
\ll \epsilon'\sum_{\substack{N\leq M\leq \sqrt{x}\\M=2^{j}}}\frac{\log \log M}{\log M}\ll \epsilon'(\log \log x)^2\ll (\log \log x)^2/N,   
\end{align*}
since we can choose $\epsilon'=\epsilon'(N)=1/N^2$, say, and now the desired estimate~\eqref{e96} follows. 
\end{proof}

\begin{proof}[Proof of Lemma~\ref{lem:sum}]
We split the sum in~\eqref{e106} into two parts: those $m$ with $D_{\leq m^{\epsilon'}}(m)>(\log m)^{2/\epsilon}$ and the remaining $m\in \mathcal{M}$. The contribution of the former part is
\begin{align}
\ll \sum_{M\leq m\leq 2M}\frac{1}{\phi(m)^2(\log m)^2}\ll \frac{1}{M(\log M)^2},    
\end{align}
which is a strong enough bound.

For the contribution of the latter part, note that all $m\in \mathcal{M}$ counted by it belong to the set
\begin{align*}
\mathcal{R}':=\{m\in [M,2M]:\, p\nmid m\, \forall p\in [(\log M)^{2/\epsilon},(2M)^{\epsilon'}],\,\,\textnormal{but}\, m'\mid m\, \textnormal{for some}\, m'\in [(\log m)^{\epsilon'},m^{\epsilon}]\cap \mathbb{P}\}.    
\end{align*}
Therefore, their contribution is
\begin{align*}
\ll \frac{1}{(\log M)^{\epsilon\cdot\epsilon'}}\sum_{\substack{M\leq m\leq 2M\\m\in \mathcal{R}'}}\frac{1}{\phi(m)^2},    
\end{align*}
and the same argument as in the proof of Lemma~\ref{lem:sparse} gives
\begin{align*}
\sum_{\substack{M\leq m\leq 2M\\m\in \mathcal{R}'}}\frac{1}{\phi(m)^2}\ll_{\epsilon} \frac{\log \log M}{M(\log M)},    
\end{align*}
so the claim follows. 
\end{proof}

Combining the estimates for the large and medium greatest common divisors,~\eqref{equ8} follows, and this was enough to prove Theorem~\ref{thm:main}.

\section{Proof of Hypothesis~\ref{hyp} under the pair correlation conjecture}
\label{sec:PCC}

In this section, we prove Theorem~\ref{thm:pcc} to the effect that PCC implies Hypothesis~\ref{hyp}. We begin with a few preliminary lemmas. Note that in Hypothesis~\ref{hyp} the parameter $m'$ is prime, but that hypothesis will only be used in~\eqref{e98} and~\eqref{e99} below (and it would suffice to assume that $m'$ does not have abnormally many prime factors), and the lemmas before the proof do not assume $m'$ to be a prime.

\subsection{Lemmas on Kummer-type extensions}

We first present a standard lemma (see e.g.~\cite[Proposition 3.10]{jarviniemi}).

\begin{lemma}
\label{lem:indBound}
Let $a$ and $b$ be multiplicatively independent integers. We have
$$(m')^2\phi(m) \ll_{a, b} [\mathbb{Q}(\zeta_m, a^{1/m'}, b^{1/m'}) : \mathbb{Q}] \le (m')^2\phi(m)$$
and
$$m'\phi(m) \ll_{a, b} [\mathbb{Q}(\zeta_m, (ba^{-r})^{1/m'}) : \mathbb{Q}] \le m'\phi(m)$$
for all positive integers $r$, $m$ and $m'$.
\end{lemma}

In what follows, we denote
\begin{align*}
K:&=\mathbb{Q}(\zeta_{m},a^{1/m'},b^{1/m'})\\
K_r:&=\mathbb{Q}(\zeta_{m},(ba^{-r})^{1/m'}),    
\end{align*}
where $m'\mid m$. We bound the number of conjugacy classes in $\Gal(K_r/\mathbb{Q})$.

\begin{lemma}
\label{lem:conFix}
For each $r$, the number of conjugacy classes of $\Gal(K/\mathbb{Q})$ which contain elements fixing $K_r$ is at most $\tau(m')$. Furthermore, the size of each such conjugacy class is at most $m'$.
\end{lemma}

\begin{proof}
To each element $\sigma \in G := \Gal(K/\mathbb{Q})$, we associate a triple $(X,Y,Z)\in (\mathbb{Z}/m\mathbb{Z})^{\times}\times (\mathbb{Z}/m'\mathbb{Z})^2$ such that $\sigma(\zeta_m) = \zeta_m^X$, $\sigma(a^{1/m'}) = \zeta_{m'}^Ya^{1/m'}$ and $\sigma(b^{1/m'}) = \zeta_{m'}^Zb^{1/m'}$. We write $\sigma \sim (X, Y, Z)$. Consider the set of such tuples with a binary operation $\ast$ given by
$$(X, Y, Z)\ast (A, B, C) = (AX, Y+BX, Z+CX).$$
This operation is clearly associative. The identity of $\ast$ is $(1,0,0)$. The inverse of $(X, Y, Z)$ under $\ast$ is $(X^{-1}, -YX^{-1}, -ZX^{-1})$, where $x^{-1}$ is the inverse of $x$ modulo $m$. Thus the set of these triples with operation $\ast$ forms a group. The conjugacy class containing $(A, B, C)$ is given by the set of elements of the form
$$(X, Y, Z)\ast (A, B, C)\ast (X^{-1}, -YX^{-1}, -ZX^{-1}) = (A, BX-(A-1)Y, CX-(A-1)Z),$$ 
where $(X,Y,Z)\in \Gal(K/\mathbb{Q})$ varies.

Let $\sigma \sim (A, B, C)$ be an element of $\Gal(K/\mathbb{Q})$ fixing $K_r$. Now $A = 1$ and $C \equiv rB \pmod{m'}$. Consider for variable $(X, Y, Z) \in \Gal(K/\mathbb{Q})$ the conjugation
$$(X, Y, Z)\ast (1, B, rB)\ast (X^{-1}, -YX^{-1}, -ZX^{-1}) = (1, XB, rXB).$$
Since $X\in (\mathbb{Z}/m\mathbb{Z})^{\times}$ varies, we see that there are at most $\tau(m')$ conjugacy classes corresponding to elements fixing $K_r$, one for each integer $d\mid m$ such that there exists a map $\sigma \sim (1, B, rB)$ satisfying $(m, B) = d$. The sum of the sizes of these conjugacy classes is at most $m'$, so each single class is of size at most $m'$.
\end{proof}

Note that the proof of Lemma~\ref{lem:conFix} gives that if some element of a conjugacy class of $\Gal(K/\mathbb{Q})$ fixes $K_r$, then all of the elements of it do.

We also need a bound for the number of conjugacy classes in $\Gal(K/\mathbb{Q})$.

\begin{lemma}
\label{lem:conAll}
The number of conjugacy classes in $\Gal(K/\mathbb{Q})$ is $\ll_{a, b} m\tau(m')^2$.
\end{lemma}

\begin{proof}
Recall the notation of the proof of Lemma~\ref{lem:conFix}. Let $H$ be the set of all triplets of the form $(X, Y, Z) \in (\mathbb{Z}/m\mathbb{Z})^{\times}\times (\mathbb{Z}/m'\mathbb{Z})^2$ equipped with the operation $\ast$, so $G := \Gal(K/\mathbb{Q})$ is a subgroup of $H$. By Lemma~\ref{lem:indBound} the index of $G$ in $H$ is bounded. This implies by~\cite{gallagher} that the number of conjugacy classes of $G$ is at most a  constant times that of $H$. We therefore consider only the group $H$ from now on. In what follows, we identify the elements of $\mathbb{Z}/k\mathbb{Z}$ with integers in $\{1, 2, \ldots , k\}$.

For each $A \in (\mathbb{Z}/m\mathbb{Z})^{\times}$ let $g = (A-1, m)$ and let $S_A$ be the set of elements $(a, b, c) \in H$ with $a = A, b \mid g$ and $1 \le c \le g$. Let $S$ be the union of $S_A$ over all $A$. We claim that each element in $H$ is conjugate to at least one element in $S$.

Let $(A, B, C) \in H$ be given. As in the proof of Lemma~\ref{lem:conFix}, the conjugacy class of $(A, B, C)$ consists of elements of the form
$$(A, BX - (A-1)Y, CX - (A-1)Z).$$
There exists an element $X \in (\mathbb{Z}/m\mathbb{Z})^{\times}$ such that $BX \pmod{g}$ divides $g$. Fix such an element $X$. We may now choose $Y$ and $Z$ such that $BX - (A-1)Y \pmod{m'}$ divides $g$ and $CX - (A-1)Z \pmod{m'}$ belongs to $\{1, 2, \ldots , g\}$. For this choice of $X, Y$ and $Z$ we have $(A, BX - (A-1)Y, CX - (A-1)Z) \in S_A$, proving the claim.

The result follows by
$$|S| = \sum_{A \in (\mathbb{Z}/m\mathbb{Z})^{\times}} |S_A| \le \sum_{A = 1}^m \tau((A-1, m'))(A-1, m') = \sum_{g \mid m'} \frac{m}{m'}\phi(m'/g)\tau(g)g  \le m\tau(m')^2.$$
\end{proof}

\begin{lemma}\label{lem:conjugacy}
Let $C_{r, 1}, \ldots , C_{r, m_r}$ be the conjugacy classes of $\Gal(K/\mathbb{Q})$ which fix $K_r$ but which do not fix any of the fields $\mathbb{Q}(\zeta_{q}, a^{1/q}) \subset K$, where $q$ ranges over the prime divisors of $m'$. Then the $C_{r, i}$ are pairwise disjoint.
\end{lemma}

\begin{proof}
Assume that $\sigma \in C_{r_1, i_1} \cap C_{r_2, i_2}$ for some $1\leq r_1<r_2\leq m'$. Thus, $\sigma$ fixes both
$$K_{r_1} = \mathbb{Q}(\zeta_{m}, (ba^{-r_1})^{1/m'})$$
and
$$K_{r_2} = \mathbb{Q}(\zeta_{m}, (ba^{-r_2})^{1/m'}).$$
In particular, $\sigma$ fixes $a^{(r_1 - r_2)/m'}$. Let $g = (r_1 - r_2, m')$ and let $e \in \mathbb{Z}$ be such that $e(r_1 - r_2) \equiv g \pmod{m'}$. We see that $\sigma$ fixes
$$(a^{(r_1 - r_2)/m'})^e = q \cdot a^{g/m'},$$
where $q \neq 0$ is some rational number, and thus $\sigma$ fixes $a^{g/m'}$.

 Since $\sigma$ fixes both $a^{m'/m'}=1$ and $a^{g/m'}$, we see similarly as above that $\sigma$ also fixes $a^{(m', g)/m'}$. Noting that $(m', g)$ is a proper divisor of $m'$, we deduce that $\sigma$ fixes $a^{1/q}$ for some prime $q \mid m'$. Finally, note that $\sigma$ fixes $\mathbb{Q}(\zeta_{q}) \subset \mathbb{Q}(\zeta_{m'}) \subset K_{r_1}$. Thus, $\sigma$ fixes $\mathbb{Q}(\zeta_{q}, a^{1/q})$, which is a contradiction with the definition of $C_{r,i}$.
\end{proof}

\subsection{A Chebotarev-type estimate in the mean square}

Having established these preliminary lemmas we proceed with deriving the implication of PCC needed. To begin we need a further conjecture, the Artin conjecture (AC) on $L$-functions.

\textbf{AC:} The the Artin $L$-functions associated with a field extension $K/\mathbb{Q}$ are analytic in the whole complex plane.

Although Artin's conjecture remains open in general, we will see in a moment that it can be proved for fields $K$ that occur in our setting, so it is not a restricting assumption.

\begin{lemma}[Murty--Murty--Wong]\label{lem:murty}
Let $K/\mathbb{Q}$ be a field extension whose Artin $L$-functions satisfy GRH, AC, and PCC. For a conjugacy class $C$ of $G=\Gal(K/\mathbb{Q})$, define the Chebychev function
\begin{align*}
\psi_C(x):=\sum_{\substack{p^j\leq x\\(\frac{K/\mathbb{Q}}{p})= C}}\log p.
\end{align*}
Then we have 
\begin{align}\label{equ13}
\sum_{C}\frac{1}{|C|}\left(\psi_C(x)-\frac{|C|}{|G|}x\right)^2\ll x(\log x)\log^2\left([K:\mathbb{Q}]x\prod_{p\in P(K/\mathbb{Q})}p \right)\frac{|\textnormal{Irr}(G)|}{|G|},    
\end{align}
where $P(K/\mathbb{Q})$ is the set of rational primes that are ramified in $K$ and $\textnormal{Irr}(G)$ is the set of irreducible representations of $G$.
\end{lemma}
\begin{proof}
This is~\cite[Theorem 5.1]{murty-wong} with the choices $m_{\chi}=\chi(1)$, $c_{\chi}=\chi(1)^{-1}$, $r=1$ and $k=\mathbb{Q}$ that correspond to our assumptions.
\end{proof}

Although Artin's conjecture remains unsolved for general $L$-functions, the next lemma says that AC holds for all Kummer-type extensions.

\begin{lemma} Hypothesis AC holds whenever $K/\mathbb{Q}$ is a Kummer-type extension.
\end{lemma}

\begin{proof}
It is known (see~\cite{wong}) that Artin's conjecture holds for the field extension $K/\mathbb{Q}$ whenever its Galois group $G:=\Gal(K/\mathbb{Q})$ is metabelian, meaning that it has a normal subgroup $A$ such that both $A$ and $G/A$ are Abelian. Thus, for any given Kummer-type extension field $K=\mathbb{Q}(\zeta_m,a_1^{m_1},\ldots, a_k^{m_k})$, it suffices to find a subfield $K_0$ such that $K/K_0$ and $K_0/\mathbb{Q}$ are both Abelian extensions and $\Gal(K_0/\mathbb{Q})$ is normal in $\Gal(K/\mathbb{Q})$. We choose $K_0=\mathbb{Q}(\zeta_m)$. Then $\Gal(K_0/\mathbb{Q})$ is cyclic, so certainly Abelian. To show that $K/K_0$ is also Abelian, it suffices to show that each $L_i:=\mathbb{Q}(\zeta_m,a_i^{m_i})/\mathbb{Q}(\zeta_m)$ is Abelian, since the compositum of Abelian extensions is also Abelian. Again, $L_i$ is a cyclic extension, and therefore Abelian. What remains to be shown then is that $A:=\Gal(K_0/\mathbb{Q})$ is a normal subgroup of $G=\Gal(K/\mathbb{Q})$. As in the proof of Lemma~\ref{lem:conFix}, for each $\sigma\in G$ we write \begin{align*}
\sigma\sim (e,e_1,\ldots, e_n)\in (\mathbb{Z}/m\mathbb{Z})^{\times}\times \prod_{i\leq n}(\mathbb{Z}/m_i\mathbb{Z}),    
\end{align*}
where this tuple is uniquely determined by the conditions $\sigma(\zeta_m)=\zeta_m^{e}$ and $\sigma(a_i^{1/m_i})=\zeta_m^{e_i}a_i^{1/m_i}$ for all $i\leq k$. The relation $\sim$ is an equivalence relation on $G$. If $\sigma'\in A$ is arbitrary, we may write $\sigma'\sim (a,0,\ldots, 0)$ for some $a\in (\mathbb{Z}/m\mathbb{Z})^{\times}$. Now $\sigma^{-1}\sim (e^{-1}, -e_1e^{-1},\ldots,-e_ne^{-1})$, so $\sigma \sigma'\sigma^{-1}=(a,0,\ldots, 0)=\sigma'$. Since $\sigma\in G$, $\sigma'\in A$ are arbitrary, the claim follows.
\end{proof}

\begin{proposition}
\label{prop:murty-wong} Assume GRH and PCC. For $K=\mathbb{Q}(\zeta_{m},a^{1/m'},b^{1/m'})$, we have
\begin{align*}
\sum_{C}\frac{1}{|C|}\left(\pi_{C}(x)-\frac{|C|}{|G|}\Li(x)\right)^2\ll_{a,b,h} x(\log x)^3\frac{\tau(m')^2}{m'\phi(m)}.    
\end{align*}
\end{proposition}

\begin{proof}
It is well known that $|\textnormal{Irr(G)}|$ equals to the number of conjugacy classes in $G$, which by Lemma~\ref{lem:conAll} is $\ll m\tau(m')^2$. Further, we have $|G|=[K:\mathbb{Q}]\gg (m')^2\phi(m)$ by Lemma~\ref{lem:indBound}. Lastly, note that if $p\in P(K/\mathbb{Q})$, then $p\mid \textnormal{disc}(K:\mathbb{Q})$ and that by~\cite[page 490]{cooke} we have
\begin{align*}\textnormal{disc}(K:\mathbb{Q})\mid (\textnormal{disc}(\mathbb{Q}(\zeta_m):\mathbb{Q})\cdot \textnormal{disc}(\mathbb{Q}(a^{1/m}):\mathbb{Q})\cdot \textnormal{disc}(\mathbb{Q}(b^{1/m}):\mathbb{Q}))^{R}
\end{align*}
for some $R$. A computation of the discriminants shows that
\begin{align*}
p&\mid \textnormal{disc}(\mathbb{Q}(\zeta_{m}):\mathbb{Q})\Longrightarrow p\mid m\\
p&\mid \textnormal{disc}(\mathbb{Q}(a^{1/m}):\mathbb{Q})\Longrightarrow p\mid a\\
p&\mid \textnormal{disc}(\mathbb{Q}(b^{1/m}):\mathbb{Q})\Longrightarrow p\mid b,
\end{align*}
so
\begin{align*}
\prod_{p\in P(K/\mathbb{Q})}p\ll_{a,b} m\leq x.    
\end{align*}
In conclusion, we have the bound
\begin{align}\label{equ16}
\sum_{C}\frac{1}{|C|}\left(\psi_C(x)-\frac{|C|}{|G|}x\right)^2\ll x(\log x)^3\frac{\tau(m')^2}{m'\phi(m)}.    
\end{align}
To pass from $\psi_C(x)$ to $\pi_C(x)$, we use partial summation. Note that
\begin{align*}
\psi_C(x)=\sum_{\substack{p^j\leq x\\(\frac{K/\mathbb{Q}}{p})=C}}\log p= \sum_{\substack{p\leq x\\(\frac{K/\mathbb{Q}}{p})=C}}\log p+O(\sqrt{x}),  
\end{align*}
so by partial summation
\begin{align*}
\pi_C(x)=\frac{\psi_C(x)}{\log x}+\int_{2}^x \frac{\psi_C(t)}{t(\log t)^2}\, dt+O(\sqrt{x}).    
\end{align*}
This gives
\begin{align*}
\pi_C(x)-\frac{|C|}{|G|}\Li(x)=\frac{\psi_C(x)-\frac{|C|}{|G|}x}{\log x}+\int_{2}^x \frac{\psi_C(t)-\frac{|C|}{|G|}\Li(t)}{t(\log t)^2}\, dt+O(\sqrt{x}),    
\end{align*}
and further 
\begin{align}\label{equ14}
\left(\pi_C(x)-\frac{|C|}{|G|}\Li(x)\right)^2\ll\left(\frac{\psi_C(x)-\frac{|C|}{|G|}x}{\log x}\right)^2+\left(\int_{2}^x \frac{\psi_C(t)-\frac{|C|}{|G|}\Li(t)}{t(\log t)^2}\, dt\right)^2+x.    
\end{align}
By Cauchy--Schwarz, we can bound
\begin{align}\label{equ15}\begin{split}
\left(\int_{2}^x \frac{\psi_C(t)-\frac{|C|}{|G|}\Li(t)}{t(\log t)^2}\, dt\right)^2&\leq \left(\int_{2}^x\frac{dt}{t(\log t)^2}\right)\left(\int_{2}^x \frac{(\psi_C(t)-\frac{|C|}{|G|}\Li(t))^2}{t(\log t)^2}\, dt\right)\\
&\ll \int_{2}^x \frac{(\psi_C(t)-\frac{|C|}{|G|}\Li(t))^2}{t(\log t)^2}\, dt. 
\end{split}
\end{align}
Now, multiplying~\eqref{equ14} by $1/|C|$, summing over $C$ and using~\eqref{equ15},~\eqref{equ16}, we get
\begin{align*}
\sum_{C}\frac{1}{|C|}\left(\pi_C(x)-\frac{|C|}{|G|}\Li(x)\right)^2\ll_{a,b} x(\log x)^3\frac{\tau(m')^2}{m'\phi(m)}.    
\end{align*}
as desired.
\end{proof}

We are now ready to prove Hypothesis~\ref{hyp} under PCC.

\begin{proof}[Proof of Theorem~\ref{thm:pcc} assuming PCC] 
We may assume that we are in the regime $y\leq m^{10}$, say, since otherwise the GRH-conditional estimate in~\eqref{e104} is good enough.

Recall that $C_{r, 1}, \ldots , C_{r, m_r}$ are the conjugacy classes of $\Gal(K/\mathbb{Q})$ which fix $K_r$ but which do not fix any of the fields $\mathbb{Q}(\zeta_{q}, a^{1/q}) \subset K$, where $q$ ranges over the prime divisors of $m'$. By Lemma~\ref{lem:conFix}, we have $|C_{r, i}| \le m'$ and $m_r \le \tau(m')$ for all $r$ and $i$. Let $\pi_{C_{r, i}}(x)$ denote the number of primes $p \le x$ such that the Artin symbol of $p$ with respect to $K$ is $C_{r, i}$. We now have
\begin{align*}
&\sum_{r=0}^{m'-1} |\{p \le y : p \,\,\textnormal{splits completely in}\,\, \mathbb{Q}(\zeta_m, (ba^{-r})^{1/m'})\}|^2\\
&=\sum_{r \pmod{m'}} \left(\sum_{i = 1}^{m_r}\pi_{C_{r, i}}(y)\right)^2\\
&\ll  \sum_{r \pmod{m'}} \left(\sum_{i = 1}^{m_r}\left( \pi_{C_{r, i}}(y) - \Li(y)\frac{|C_{r, i}|}{[K : \mathbb{Q}]}\right)\right)^2+\sum_{r \pmod{m'}} \left( \Li(y)\frac{|C_{r, i}|}{[K : \mathbb{Q}]}\right)^2.  
\end{align*}
By Lemmas~\ref{lem:indBound} and~\ref{lem:conFix}, and the assumption that $m'$ is prime, the latter sum is 
\begin{align}\label{e98}
\ll \frac{\tau(m')^2y^2}{m'\phi(m)^2(\log y)^2}\ll_{\epsilon} \frac{y^2}{\phi(m)^2(\log y)^2m'}.    
\end{align}
When it comes to the first sum, we estimate using Cauchy--Schwarz that
\begin{align*}
&\sum_{r \pmod{m'}} \left(\sum_{i = 1}^{m_r}\left( \pi_{C_{r, i}}(y) - \Li(y)\frac{|C_{r, i}|}{[K : \mathbb{Q}]}\right)\right)^2\\ 
&\ll \sum_{r \pmod{m'}} m_r \sum_{i = 1}^{m_r} \Big(\pi_{C_{r, i}}(y) - \Li(y)\frac{|C_{r, i}|}{[K : \mathbb{Q}]}\Big)^2 \\
\nonumber&\ll \tau(m') \sum_{r \pmod{m'}} \sum_{i = 1}^{m_r} \Big(\pi_{C_{r, i}}(y) - \Li(y)\frac{|C_{r, i}|}{[K : \mathbb{Q}]}\Big)^2 \\
&\ll \tau(m')m' \sum_{r \pmod{m'}} \sum_{i = 1}^{m_r} \frac{1}{|C_{r, i}|} \Big(\pi_{C_{r, i}}(y) - \Li(y)\frac{|C_{r, i}|}{[K : \mathbb{Q}]}\Big)^2.
\end{align*}

Note then that the conjugacy classes are disjoint by Lemma~\ref{lem:conjugacy}. Hence, Proposition~\ref{prop:murty-wong} produces for the previous expression a bound of
\begin{align}\label{e99}\begin{split}
&\ll   \frac{\tau(m')^3}{\phi(m)}y(\log y)^3\ll \frac{y^2}{\phi(m)^2(\log y)^2}\cdot \frac{m(\log y)^5}{y}\\
&\ll \frac{y^2}{\phi(m)^2(\log y)^2}\max\left\{\frac{1}{(\log y)^2},\frac{1}{m'}\right\},  
\end{split}
\end{align}
since $y\geq m\exp((\log m)^{1/2})$, $\phi(m)\gg m/\log \log m$, $\tau(m')= 2$ and $m'\leq m^{\epsilon}$. Combining~\eqref{e98} and~\eqref{e99}, the claim follows.
\end{proof}

\section{Almost prime values of shifted exponentials}

In this section, we prove Theorem~\ref{thm:almostprime}. The proof follows along the same lines as the proof of Theorem~\ref{thm:main}.

\begin{proof}[Proof of Theorem~\ref{thm:almostprime}] Firstly, we deal with the simple cases $|b|\leq 1$. The case $b=0$ is trivial. For the cases $b=\pm 1$, we observe that by the Hardy--Ramanujan theorem almost all $n\leq x$ have $(1+o(1))\log \log x$ prime divisors, and if $p\mid n$ is odd, then $a^{p}\pm 1\mid a^n\pm 1$ for either choice of the sign $\pm$. To conclude, we apply a theorem of Zsigmondy~\cite{zsigmondyprimes}, according to which each term in the sequence $a^n-1$ for $n\geq 7$ (and also for $n\leq 6$ with a few exceptions) has a primitive divisor, that is, a prime factor $q\mid a^n-1$ that does not divide $a^m-1$ for any $1\leq m<n$.  

Now let $|b|>1$. Let the sets $A_p$ be defined as in Section~\ref{sec:chungBound}. In the course of proving Theorem~\ref{thm:main}, we in fact proved more strongly that
\begin{align*}
\Pr\left(\left\{n\leq x: \Big|\sum_{p\in S_{\leq \sqrt{x}}}1_{A_p}(n) -\mu\big|\geq \varepsilon \mu\right\}\right)=o_{\varepsilon}((\log \log x)^2),
\end{align*}
where 
\begin{align*}
 \mu :=\sum_{p\in S_{\leq \sqrt{x}}}\Pr(A_p)=(c+o(1))\log \log x   
\end{align*}
for some constant $c=c_{a,b}>0$. Indeed, this follows from Lemma~\ref{lem:2moment} and formulas~\eqref{equ5} and~\eqref{equ6}. The claim is now proved.
\end{proof}

\section{The case of prime exponents}

We turn to the proof of Theorem~\ref{thm:primeexp}. Most steps go through similarly as in the case of Theorem~\ref{thm:main}, but there are a few extra complications, namely the set $S$ from Section~\ref{sec:equiLog} needs to be defined differently, using the ``$W$-trick'', and therefore we will also need Lemma~\ref{lem:eulerphi} below.

\begin{proof}[Proof of Theorem~\ref{thm:primeexp}] 

Let $h_a$ and $h_b$ be the largest integers such that $a$ (respectively $b$) is a perfect power\footnote{Note that here we can no longer assume that $a$ is not a perfect power.} of order $h_a$ (respectively $h_b$). Define $h = h_ah_b$. Let $W=\prod_{p\leq w,p\nmid 2abh}p$, where $w$ is a large parameter. We redefine the set $S$ in Section~\ref{sec:equiLog} as
\begin{align*}
S:&=\{p\in \mathbb{P}:\,\, p\equiv 1 \pmod{|4abh|},\,\, p\equiv 2\pmod W,\,\, \ord_p(a)=\ord_p(b)=\frac{p-1}{2h}\},    
\end{align*}
As before, define $S_{m,d,r}$ by~\eqref{equ17}. We observe then that for $p\in S$ we have $(\ell(p),\ord_p(a))=1$. Indeed, if this did not hold, we would have
\begin{align*}
\ord_p(a^{\ell(p)})\leq \frac{\ord_p(a)}{2}=\frac{\ord_p(b)}{2}<\ord_p(b), 
\end{align*}
which is absurd as $a^{\ell(p)}\equiv b\pmod p$.

We begin by proving that $d_{\mathbb{P}}(S)$ exists and is positive, and furthermore that $d_{\mathbb{P}}(S_{m,d,r})$ is independent of $r$ as long as $r$ is coprime to $m$. We remark that Lenstra's results in~\cite{lenstra} do not immediately apply to our situation. In~\cite{lenstra}, the results concern the order of the reduction of a single multiplicative subgroup of $\mathbb{Q}^{\times}$ modulo primes, while in our case we wish to control simultaneously the order of both $a$ and $b$. However, Lenstra's methods may be adapted to this more general setting of two multiplicative subgroups, the details of the proofs being given in~\cite{jarviniemi}. We thus have analogues of Lenstra's results in this setting.

Similarly as in Section~\ref{sec:equiLog}, define
$$F_{m, d, r} := \mathbb{Q}(\zeta_{|4abh|}, \zeta_{W}, a^{1/2h}, b^{1/2h}, (ba^{-r})^{1/2hm}),$$
$$L_{\ell, a} := \mathbb{Q}(\zeta_{q(\ell)}, a^{1/q(\ell)})$$
and
$$L_{\ell, b} := \mathbb{Q}(\zeta_{q(\ell)}, b^{1/q(\ell)}),$$
where $q(\ell)$ is the smallest power of $\ell$ not dividing $2h$, and let $L_{n, a}$ (respectively $L_{n, b}$) denote the compositum of $L_{\ell, a}$ (respectively $L_{\ell, b}$) over the primes $\ell \mid n$. Now $p \in S$ if and only if the Artin symbol of $p$ with respect to $F_{1, 0, 0}$ belongs to a suitable conjugacy class $C$ of $\Gal(F_{1, 0, 0}/\mathbb{Q})$ (namely the one fixing
$$F_{1, 1, 0}' := \mathbb{Q}(\zeta_{|4abh|}, a^{1/2h}, b^{1/2h})$$
and mapping $\zeta_W \to \zeta_W^2$) and $p$ does not split completely in any of $L_{\ell, a}$ and $L_{\ell, b}$ for any prime $\ell$. If $a_n$ is defined as the local density of primes $p$ satisfying this condition for all $\ell \mid n$, by the analogue of Lenstra's result~\eqref{equ4.1} it suffices to show that the local densities $a_n \neq 0$ for all $n$ in order to have $d_{\mathbb{P}}(S) \neq 0$. 

Proceeding as in Section~\ref{sec:equiLog}, one has to check two conditions. The first condition is $a_1 \neq 0$, i.e. that $C \neq \emptyset$. To do this one notes that the largest Abelian subfield of $F_{1, 1, 0}'$ does not intersect  $\mathbb{Q}(\zeta_W)$ nontrivially, and thus $[F_{1, 1, 0}'(\zeta_W) : F_{1, 1, 0}'] = \phi(W)$, implying that 
\begin{align} \label{equ18}
\Gal(F_{1, 1, 0}/\mathbb{Q}) \simeq \Gal(F_{1, 1, 0}'/\mathbb{Q}) \times \Gal(\mathbb{Q}(\zeta_W)/\mathbb{Q}).
\end{align}
It follows that $|C| = 1$.

The second condition to check is that for any prime $\ell$ and squarefree $n$ we have
$$[K'(a^{1/q(\ell)}, b^{1/q(\ell)}) : K'] = \ell^2,$$
where
$$K' := \mathbb{Q}(\zeta_{|4abh|}, \zeta_W, a^{1/2h}, b^{1/2h}, a^{1/q(n/\ell)}, b^{1/q(n/\ell)}).$$
We do this by showing, more generally, that for any positive integers $V, s$ and $t$ satisfying $4abh, s, t \mid V$ and $2h \mid s, t$ we have
\begin{align}\label{equ19}
[\mathbb{Q}(\zeta_V, a^{1/s}, b^{1/t}) : \mathbb{Q}] = \phi(V)\frac{s}{2h_a}\frac{t}{2h_b}
\end{align}
(cf. equation~\eqref{equ4.2}). The upper bound implicit in the result is easy to establish. For the lower bound, apply the result in~\cite{garrett} to the numbers of the form $a^{e/s}b^{f/t}, 0 \le  e < s/2h_a, 0 \le f < t/2h_b$ over $\mathbb{Q}(\zeta_V)$. We want to check that
$$a^{e/s}b^{f/t} \in \mathbb{Q}(\zeta_V)$$
for such $e$ and $f$ implies that $e = f = 0$. Similarly to the proof of~\eqref{equ4.2}, we obtain the divisibility relations
$$s \mid 2h_ae, t \mid 2h_bf,$$
from which the result follows.

Having established $d_{\mathbb{P}}(S) \neq 0$ we now proceed to proving that $d_{\mathbb{P}}(S_{m, d, r})$ is independent of $r$ (for $(r, m) = 1$). By a variant of Lenstra's product formula in~\eqref{equ4.1} for multiple variables one has
$$d_{\mathbb{P}}(S_{m, d, r}) = a_n \prod_{\substack{\ell \nmid n  \\ \ell \in \mathbb{P}}} \left(1 - \frac{C(\ell)}{[L_{\ell, a}L_{\ell, b} : \mathbb{Q}]}\right),$$
where $C(\ell)$ is the number of conjugacy classes of $\Gal(L_{\ell, a}L_{\ell, b}/\mathbb{Q})$ fixing at least one of $L_{\ell, a}$ and $L_{\ell, b}$. Since the terms of the infinite product do not depend on $r$, it suffices to check that the local densities $a_n = a_n(r)$ do not depend on $r$. 

Clearly for any prime $\ell \mid W$ and $n$ not divisible by $\ell$ we have $a_n(r) = a_{n/\ell}(r)$, as the condition $p \equiv 2 \pmod{W}$ guarantees that $p$ does not split completely in $\mathbb{Q}(\zeta_{\ell})$ and thus not in $L_{\ell, a}$ or $L_{\ell, b}$. We may therefore assume that $(n, W) = 1$.

Similarly as one proves $|C| = 1$ via~\eqref{equ18}, one obtains
\begin{align*}
\Gal(F_{m, d, r}L_{k_a, a}L_{k_b, b}/\mathbb{Q}) = \Gal(F_{m, d, r}'L_{k_a, a}L_{k_b, b}/\mathbb{Q}) \times \Gal(\mathbb{Q}(\zeta_W)/\mathbb{Q})
\end{align*}
for any $k_a, k_b \mid n$, where
$$F_{m, d, r}' := \mathbb{Q}(\zeta_{|4abh|}, a^{1/2h}, b^{1/2h}, (ba^{-r})^{1/2hm}).$$
We thus see that there is exactly one element $\sigma \in \Gal(F_{m, d, r}L_{k_a, a}L_{k_b, b}/\mathbb{Q})$ which fixes $L_{k_a, a}L_{k_b, b}$ and whose restriction to $F_{m, d, r}$ belongs to $C$. This leads to the following analogue of~\eqref{equ4.1.5}:
$$a_n(r) = \frac{1}{[F_{m, d, r}L_{n, a}L_{n, b} : \mathbb{Q}]}\sum_{k_a \mid n} \sum_{k_b \mid n} \mu(k_a)\mu(k_b)[F_{m, d, r}L_{k_a, a}L_{k_b, b} : \mathbb{Q}].$$

To conclude the proof one proves an analogue of~\eqref{equ4.2} and~\eqref{equ19}. This time one claims that
$$[\mathbb{Q}(\zeta_V, a^{1/s}, b^{1/t}, (ba^{-r})^{1/u}) : \mathbb{Q}] = \phi(V)\frac{s}{2h_a}\frac{t}{2h_b}\frac{u}{(s, t, u)},$$
when $4abh, s, t, u \mid V, 2h_a \mid s, 2h_b \mid t, 2h_ah_b \mid u$ and $(r, u) = 1$. As before, the implied upper bound is easy. For the lower bound, we once again apply~\cite{garrett} and consider the numbers of the form
$$a^{e/s}b^{f/t}(ba^{-r})^{g/u},$$
where $0 \le e < s/2h_a, 0 \le f < t/2h_b$ and $0 \le g < u/(s, t, u)$. Assume that such a number belongs to a cyclotomic field. Define
$$C := \left(a^{e/s}b^{f/t}(ba^{-r})^{g/u}\right)^{stu} = a^{etu - rgst}b^{sfu + stg}.$$
The integer $C$ must be a perfect power of order $stu/2$, and so
\begin{align}\label{equ19.1}
stu \mid 2h_a(etu - rgst)
\end{align}
and
\begin{align}\label{equ19.2}
stu \mid 2h_b(sfu + stg).
\end{align}
The rest is elementary. As $2h_a \mid s$, the first condition gives us
$$2h_atu \mid 2h_argst,$$
and similarly from the second condition
$$2h_bsu \mid 2h_bstg.$$
We now have, using $(r, u) = 1$, the relation $u \mid g(s, t)$ and thus $u/(s, t, u) \mid g$. This implies $g = 0$. Plugging this into~\eqref{equ19.1} and~\eqref{equ19.2} gives $s \mid 2h_ae$ and $t \mid 2h_bf$, and thus $e = f = 0$.

Thus all of the results of Section~\ref{sec:equiLog} work with this new definition of $S$. When it comes to Section~\ref{sec:chungBound}, we redefine $\textnormal{Pr}$ as the uniform probability measure on $[1,x]\cap \mathbb{P}$. We also redefine
\begin{align*}
A_p:=\{p'\in [1,x]\cap \mathbb{P}:\,\, p'\equiv \ell(p)\pmod{\ord_p(a)},\, p'\neq \ell(p)\}.     
\end{align*}
The second moment method (Lemma~\ref{lem:2moment}) gives an upper bound
\begin{align*}
 \ll \frac{\sum_{p, q \in S_{\leq \log x}} \Pr(A_p \cap A_q)}{\left(\sum_{p \in S_{\leq \log x}} \Pr(A_p)\right)^2}-1
\end{align*}
for the probability that $a^p-b$ is prime with $p\leq x$. By the Siegel--Walfisz theorem and the fact that $p\in S$ implies $(\ell(p),\ord_p(a)) = 1$, we see that the previous expression is
\begin{align}\label{equ23}
=\frac{\sum_{p, q \in S_{\leq \log x}}1/(\phi(\lcm(\ord_p(a)\ord_q(a))))+O((\log x)^{-100}))}{\left(\sum_{p \in S_{\leq \log x}}( 1/(\phi(\ord_p(a))+O((\log x)^{-100})))\right)^2}-1.    
\end{align}
The denominator here is
\begin{align}\label{equ22}
\left(\sum_{p\in S_{\leq \log x}}\frac{1}{\phi((p-1)/2h)}\right)^2+O(1)
\end{align}

Since $\phi((p-1)/2h)/p$ fluctuates on the primes, we cannot a priori compute this sum just based on the value of $d_{\mathbb{P}}(S)$. However, using the following lemma we can compute~\eqref{equ22} in terms of $d_{\mathbb{P}}(S)$.

\begin{lemma}\label{lem:eulerphi} Let $x\geq 2$, $k\in \mathbb{N}$ $w\geq 2$, and let $W=\prod_{p\leq w, p\nmid k}p$. Then for $\lambda\geq 1$ and $Wk\leq \sqrt{x}$, $k\leq x^{0.1}$ we have
\begin{align}\label{equ20}
|\{p\leq x:\,\, p\equiv 1\pmod k,\,\, p\equiv 2\pmod{W},\,\, \frac{p-1}{\phi(p-1)}>\lambda \frac{k}{\phi(k)}\}|\ll \lambda^{-w}\frac{\pi(x)}{\phi(kW)}.    
\end{align}
\end{lemma}

We remark that much stronger tail bounds are known in the case $k=W=1$; see~\cite{weingartner} for instance. However, for us the $W$-aspect is crucial, since in particular taking $\lambda=1+w^{-0.9}$ we can conclude that all but $\ll \exp(-cw^{0.1})$-proportion of primes $p\leq x, p\equiv 1\pmod k, p\equiv 2\pmod W$ satisfy $\phi(p-1)=(1+O(w^{-0.9}))\frac{\phi(k)}{k}p$, say.

\begin{proof}
This is a standard application of the method of moments. For any $\ell\geq 1$, the left-hand side of~\eqref{equ20} is

\begin{align}\label{equ21}
\leq \lambda^{-\ell}\sum_{\substack{p\leq x\\p\equiv 1\pmod k\\p\equiv 2\pmod W}}\left(\frac{p-1}{\phi(p-1)}\frac{\phi(k)}{k}\right)^{\ell}.   
\end{align}
For $n\equiv 1\pmod k, n\equiv 2\pmod{W}$, we have
\begin{align*}
\left(\frac{n}{\phi(n)}\frac{\phi(k)}{k}\right)^{\ell}=\prod_{\substack{p\mid n\\p\nmid kW}}\left(1-\frac{1}{p}\right)^{-\ell}:=g(n),
\end{align*}
where $g(n)$ is a multiplicative function defined by $g(p^j)=(1-\frac{1}{p})^{-\ell}$ for $p\nmid kW$ and $g(p^j)=1$ for $p\mid kW$. By M\"obius inversion, we have $g(n)=\sum_{d\mid n}h(d)$, where $h$ is a multiplicative function defined by $h(p)=g(p)-1$ for $p\nmid kW$ and $h(p^j)=0$ when $j\geq 2$ or $j=1$, $p\mid kW$. 

Now the sum over $p$ in~\eqref{equ21} becomes
\begin{align}\label{e11}
\sum_{d\leq x}h(d) \sum_{\substack{p\leq x\\ p\equiv 1 \pmod{\lcm(k,d)}\\ p\equiv 2\pmod{W}}}1.       
\end{align}

By the Brun--Titchmarsh inequality, the contribution of the terms $1\leq d\leq x^{0.4}$ to the sum is
\begin{align*}
&\ll \frac{\pi(x)}{\phi(W)}\sum_{d\leq x^{0.4}}\frac{h(d)}{\phi(\lcm(k,d))}\\
&\leq \frac{\pi(x)}{\phi(kW)}\prod_{p\nmid kW}\left(1+\frac{h(p)}{\phi(p/(k,p))}\right)\\
&\ll \frac{\pi(x)}{\phi(kW)}\prod_{p\nmid kW}\left(1+\frac{(1+1/(p-1))^{\ell}-1}{p-1}\right).
\end{align*}
By the mean value theorem for derivatives, we have $(1+x)^{\ell}\leq 1+e\ell x$ for $0\leq x\leq 1/(\ell-1)$. This together with the inequality $1+x\leq e^x$ lets us bound the previous expression by 
\begin{align*}
\frac{\pi(x)}{\phi(kW)}\exp\left(e\ell\sum_{p>w}\frac{1}{(p-1)^2}\right)\ll \frac{\pi(x)}{\phi(kW)},
\end{align*}
by the prime number theorem and the choice $\ell=w$.

The contribution to~\eqref{e11} from $x^{0.4}\leq d\leq x$, in turn, is trivially
\begin{align*}
\ll x\sum_{x^{0.4}\leq d\leq x}\frac{h(d)}{\lcm(k,d)W}\leq \frac{x^{0.97}}{kW}\sum_{x^{0.4}\leq d\leq x}\frac{h(d)}{(d/(k,d))^{0.9}},
\end{align*}
and by essentially the same computation as above this is 
\begin{align*}
\ll \frac{x^{0.97}}{kW}\exp(Cw^{0.1})\ll \frac{x^{0.98}}{kW},
\end{align*}
since $w\ll \log x$. Recalling that we have a factor of $\lambda^{-\ell}$ in~\eqref{equ21}, the proof is complete.
\end{proof}

By Lemma~\ref{lem:eulerphi}, for all but $\ll \exp(-cw^{0.1})$-proportion of $p\in S$ we have $\phi(p-1)=(1+O(w^{-0.9}))\frac{\phi(|4ab|)}{|4ab|}p$, and for any $k\geq 1$, for all but $\ll \exp(-cw^{0.1})2^{-10k}$-proportion of $p\in S$ we have $\phi(p-1)\geq 2^{-k}(1-w^{-0.9})\frac{\phi(|4ab|)}{|4ab|}p$. Thus, the expression~\eqref{equ22} takes the form 
\begin{align*}
&\left(\sum_{p\in S_{\leq \log x}}\frac{1}{(p-1)/2h\cdot \frac{\phi(|4ab|)}{|4ab|}}\right)^2+o_{w\to \infty}((\log \log x)^2)\\
&=\left(4h^2\left(\frac{|4ab|}{\phi(|4ab|)}\right)^2d_{\mathbb{P}}(S)^2+o_{w\to \infty}(1)\right)(\log \log x)^2,
\end{align*}
which is the same quantity as in Section~\ref{sec:chungBound}. 

To deal with the numerator in~\eqref{equ23}, we again apply Lemma~\ref{lem:eulerphi} to write it as
\begin{align*}
\sum_{p,q\in S_{\leq \log x}}\frac{(|4ab|/\phi(|4ab|))^2 }{\lcm((p-1)/2h,(q-1)/2h)}+o_{w\to \infty}((\log \log x)^2).    
\end{align*}
Then we split the sum according to the value of $m:=((p-1)/2h,(q-1)/2h)$. We reduce to proving~\eqref{equ7} and~\eqref{equ8} (with the difference that the $p$ sum in the definition of $\Sigma_{m,d,r}$ only goes up to $\log x$ and the $r$ sum in $\Sigma_{m,d}$ only goes over $(r,m)=1$). Estimate~\eqref{equ8} is proved precisely as in the case of Theorem~\ref{thm:main}, since the set $S$ in Section~\ref{sec:equiLog} contains the set $S$ defined in this section. Estimate~\eqref{equ7} is proved similarly as in Section~\ref{sec:gcdSmall}, since the argument there is just based on Lemmas~\ref{lem:denPos} and~\ref{lem:denInv} for which we have analogues in this case, with the modification of considering only $r$ with $(r, m) = 1$ in~\eqref{equ8.1}.
\end{proof}

\section{Necessary conditions for primality of shifted exponentials}
\label{sec:necessary}

One might be tempted to conjecture that for Question~\ref{conj1} to have a positive answer it would suffice to look at the $q=1$ case there, that is, to show that $a^n-b$ has no fixed prime divisor and that $a^n-b$ does not factor as a result of a polynomial identity. These conditions are however not sufficient, as demonstrated by the sequence $29^n-4$. For $n$ even, $29^n-4$ factors as the difference of two squares, whereas for $n$ odd we have $5\mid 29^n-4$ (even though $29^n-4$ has no fixed prime divisor and $x^n-4$ is irreducible for $n$ odd). More generally, we have the following construction that shows the necessity of the ``for every $q\geq 1$ there exist $1\leq r\leq q$'' part of Question~\ref{conj1}.
\begin{proposition} Let $p$ be a prime. Then there exist integers $a,b>1$ such that 
\begin{itemize}

\item [(i)] The sequence $a^n-b$ has no fixed prime divisor;

\item [(ii)] The polynomial $x^{h}-b\in \mathbb{Z}[x]$ is irreducible, where $h$ is the largest integer such that $a$ is an $h$th power;
\end{itemize}

but for every $1\leq r\leq p$ either of the following holds:

\begin{itemize}

\item [(iii)]  The sequence $a^{pn+r}-b$ has a fixed prime divisor;

\item [(iv)] The polynomial $a^rx^{hp}-b\in \mathbb{Z}[x]$ is reducible. 
\end{itemize}
\end{proposition}

\begin{proof}
By Dirichlet's theorem, we can pick odd, distinct primes $q_1,\ldots, q_{p-1}\equiv 1\pmod{p^2}$.  Now the congruences $x^{p}\equiv 1\pmod{q_i}$ all have a solution with $x\not \equiv 1 \pmod{q_i}$, so by the Chinese remainder theorem we can find an even $a>\max\{q_1,\ldots, q_{p-1}\}$ which is not a perfect power and which satisfies $\ord_{q_i}(a)=p$ for all $i$. Then pick $c>1$ such that $c^p\equiv a^{i}\pmod{q_i}$ for all $1\leq i\leq p-1$ (this can be done, since $x^p\equiv g^{(q_i-1)/p}$ is solvable for $g$ any primitive root $\pmod{q_i}$). Since the set of suitable $c$ forms an arithmetic progression $\pmod{q_1\cdots q_{p-1}}$ and since $(q_i,a(a-1))=1$, we may additionally require that $(c,a)=1$ and $(c^p-1,a-1)=1$. Now let $b:=c^p$. 

Conditions (i) and (ii) are now obvious, as $h=1$ and $(a,b)=(a-1,b-1)=1$. Moreover, the polynomial $a^px^{hp}-b$ factorizes as the difference of two $p$th powers, and for each $1\leq i\leq p-1$ we have $a^{pn+i}-b\equiv 0\pmod{q_i}$ for all $n\geq 0$. This means that also conditions (iii) and (iv) hold. 
\end{proof}

We remark that it seems difficult to tell whether the conditions of Question~\ref{conj1} hold for a given sequence. One case where this is difficult is the case of \emph{Sierpinski numbers}: it has been conjectured by Erd\H{o}s that if $k$ is a Sierpinski number (that is, $k\cdot 2^n+1$ is composite for all natural numbers $n$), then the smallest prime divisor of $k \cdot 2^n + 1$ is bounded (see~\cite[Conjecture 2]{ffk}). This corresponds to condition (i) in Question~\ref{conj1} for $a=2$, $b=-k$ (since for large primes $p$ the congruence $k\cdot 2^n+1\equiv 0\pmod p$ is solvable if and only if $2^n+k\equiv 0\pmod p$ is solvable). In~\cite{ffk}, Filaseta, Finch and Kozek state that it is ``highly likely'' that Erd\H{o}s's conjecture is false, as the conjecture does not take into account polynomial identities, corresponding to condition (ii) in Question~\ref{conj1}. They conjecture that Erd\H{o}s's intuition is correct for all power-free $k$. They also conjecture that the smallest prime divisor of $5 \cdot 2^n + 1$ is unbounded as $n$ varies. Our proof of Theorem~\ref{thm:almostprime} gives that $\omega(5\cdot 2^n+1)\to \infty$ along almost all natural numbers $n$, but does not rule out the smallest prime factor being bounded.

\bibliography{exponentialPrimes}
\bibliographystyle{plain}

\end{document}